\documentclass[12pt]{amsart}
\usepackage[top=1in, bottom=1in, left=1in, right=1in]{geometry}
\usepackage{amsfonts}
\usepackage{amsmath}
\usepackage{comment}
\usepackage{amssymb}
\usepackage{graphicx}
\usepackage{bbm}
\usepackage{comment}
\usepackage{mathrsfs}
\numberwithin{equation}{section}
\usepackage{times}

\usepackage{siunitx}
\usepackage[usenames,dvipsnames]{color}
\usepackage{comment}
\usepackage{xcolor}
\usepackage{mathtools}
\usepackage{bm}
\usepackage{esvect}
\usepackage{hyperref}

\hypersetup{
    colorlinks = true,
linkcolor={black},
urlcolor={blue},
citecolor={blue},    
urlcolor = {blue},
citebordercolor = {0.33 .58 0.33},
 linkbordercolor = {0.99 .28 0.23},
 breaklinks=true}
 
\newcommand{\R}{\mathbb{R}}

\newcommand{\N}{\mathbb{N}}
\newcommand{\Z}{\mathbb{Z}}

\newcommand{\fanm}{\mathscr{M}}
\newcommand{\bk}{\mathbf{k}}
\newtheorem{thm}{Theorem}[section]

\newtheorem{coro}[thm]{Corollary}
\newtheorem{lem}[thm]{Lemma}
\newtheorem{rem}[thm]{Remark}

\theoremstyle{remark}

\usepackage{geometry}
\title[Generalization on Mundici's Conjecture]{Effective Estimates for a Class of Farey Fraction Sums\\ and Bounds for Mundici-Type Constants}

\author{Anji Dong, Huy Xuan Nguyen, Vi Anh Nguyen, Alexandru Zaharescu}

\address{
Anji Dong: Department of Mathematics,
University of Illinois Urbana-Champaign,
Altgeld Hall, 1409 W. Green Street,
Urbana, IL, 61801, USA}
\email{anjid2@illinois.edu}

\address{Huy Xuan Nguyen: Department of Mathematics,
University of Illinois Urbana-Champaign,
Altgeld Hall, 1409 W. Green Street,
Urbana, IL, 61801, USA}
\email{huyn4@illinois.edu}

\address{
Vi Anh Nguyen: Department of Mathematics,
University of Illinois Urbana-Champaign,
Altgeld Hall, 1409 W. Green Street,
Urbana, IL, 61801, USA}
\email{vianhan2@illinois.edu}
\address{
Alexandru Zaharescu: Department of Mathematics,
University of Illinois Urbana-Champaign,
Altgeld Hall, 1409 W. Green Street,
Urbana, IL, 61801, USA}
\email{zaharesc@illinois.edu}  
 
\begin{document}
\setcounter{tocdepth}{1}
\begin{abstract}
  Let $D_{2}(Q)$ denote the sum of squared distances between consecutive Farey fractions in the full interval $(0, 1]$. Daniele Mundici conjectured that $C(Q):=D_{2}(Q)\cdot Q^2/\log Q$ is less than 3 for all $Q\geq 2$, which is confirmed true in \cite{ViAnji}. In this paper, we generalize this result to subintervals of $(0, 1]$ and to $h$-spacings. As applications, we obtain Mundici-type bounds in these two settings, extending the full-interval consecutive-spacing case of Mundici's conjecture.
\end{abstract}
\maketitle

\section{Introduction}
Let $Q\geq 2$ be an integer. The $Q$-th Farey sequence $F_Q$ is defined as follows:
\begin{align*}
    F_Q :=\left\{\frac{a}{q}: 1\leq a\leq q\leq Q,\ (a,q)=1\right\}.
\end{align*}

Write $|F_Q|=N(Q)$, and enumerate
\begin{align*}
   F_Q :=\{\gamma_1,\gamma_2,\dots,\gamma_{N(Q)}\} 
\end{align*}
with $\gamma_1<\gamma_2<\dots<\gamma_{N(Q)}$. For simplicity, we write $N=N(Q)$, while keeping the dependence of $Q$ in mind. Moreover, we set $\gamma_0=0$. Whenever an index $j$ exceeds $N$, we use the periodic extension
\[
\gamma_j = \gamma_{j \bmod N}+\left\lfloor \frac{j}{N}\right\rfloor.
\] 
In particular, $\gamma_{j+N} = \gamma_j+1$. For general facts about Farey fractions, the reader may refer to Chapter 3 of Hardy and Wright's book \cite{Hardy} and the survey \cite{CoZa} by Cobeli and one of the authors. 

In \cite{ViAnji}, Li and two of the present authors study the distribution of spacings between consecutive Farey fractions and obtain an effective/explicit formula for $D_{2,1}(Q)$, which is the sum of squared distances between consecutive Farey fractions in the full interval $(0, 1]$. This formula is then used in \cite{ViAnji} to prove Mundici's conjectural bound for the normalized quantity
\begin{align}
    C(Q) = \frac{D_{2,1}(Q) \cdot Q^2}{\log Q},
\end{align}
which states that 
\begin{align}
C(Q)<3 \text{ for all $Q\geq 2$.}
\end{align}
We refer to this normalized quantity, and to its analogues below, as \textit{Mundici-type constants}. 
At the end of \cite{ViAnji}, the authors raise some open problems concerning two further directions: $h$-th level
consecutive spacings for $h\geq 2$, previously considered by Augustin, Boca, Cobeli, and one of the authors \cite{boca, BoCoZaha2001}, and localized sums over subintervals $I$ of $(0,1]$, a direction previously studied in \cite{BoCoZaha2001}. 

Our main goal in this paper is to tackle these open problems. These two directions introduce additional complications beyond those appearing in the full-interval consecutive-spacing problem. In the short-interval setting, we need to provide and use effective estimates for Kloosterman sums. Moreover, as established in \cite{BoCoZaha2001}, a subinterval $I$ comes with an associated quantity~$c_I$, called the \textit{defect}; whereas for the full interval $(0,1]$, this defect vanishes. On the other hand, the $h$-spacing problem involves longer chains of consecutive Farey fractions as opposed to the original problem's consecutive pairs. These chains are intrinsically related to the behavior of the so-called BCZ-map introduced by Boca, Cobeli, and one of the authors \cite{BoCoZaha2000}. For further work on this map, see, for example, Athreya and Cheung \cite{ach}. 



We first consider the $h$-spacing problem. For $h\geq 1$, define 
\begin{align}
    D_{2,h}(Q):= \sum_{j=1}^{N} (\gamma_{j+h}-\gamma_j)^2.
\end{align}
The corresponding Mundici-type constant is
\begin{align}
    C_h(Q) := \frac{D_{2,h}(Q)Q^2}{\log Q}.\label{def:Ch(Q)}
\end{align}
Our first result gives a uniform bound for $C_h(Q)$.

\begin{thm}\label{thm: Mundici conjecture}
For any integer $h\geq 1$, we have
   \begin{align}
         C_h(Q)\leq\frac{2h^2}{\log 2} \text{\quad for all $Q\geq 2$}.\label{eq: bound of Ch(Q)}
   \end{align} 
  Moreover, such an upper bound is always attained: for any $h\geq 1$, we have $C_h(2)=2h^2/\log 2$.
\end{thm}

 Notice that when $h$ grows, the bound in \eqref{eq: bound of Ch(Q)} is quadratic in $h$. However, such a bound can be significantly improved when $Q$ is large enough. With some additional work, our next result gives a linear bound in $h$ with an explicit threshold.

 \begin{thm}\label{thm: Mundici conjecture for large Q}
     For any integer $h\geq 1$, there exists an effectively computable integer $Q_h$ such that
   \begin{align}
         C_h(Q)< 3h\text{\quad for all $Q\geq Q_h$}.\label{eq:Ch(Q)<3h}
   \end{align} 
   Moreover, one can take $Q_1=2$, $Q_2 = 19397$, and for $h\geq 3$, 
   \begin{align}
Q_h = \exp\left\{\max\left\{
\frac{2D(h)}{\Delta_h},(h+1)W\left(\frac{1}{h+1}\left(\frac{h\cdot2^{3h+9}}{\Delta_h}\right)^{\frac{h+2}{h+1}}\right)\right\}\right\},
\end{align}
where the constant $D(h)$ is defined in \eqref{def:D(h)} below, $W(x)$ is the Lambert $W$ function (see \cite{Lambert}) defined as the function satisfying 
\begin{align}
    W(x)e^{W(x)} = x,
\end{align}
 and
\begin{align}
\Delta_h = 3h-\frac{12(2h-1)}{\pi^2}.
\end{align}
 \end{thm}

\begin{rem}
    We remark that the value $Q_2 = 19397$ in Theorem~\ref{thm: Mundici conjecture for large Q} is the best possible in the sense that the inequality in \eqref{eq:Ch(Q)<3h} fails if one takes $Q_2=19396$.
\end{rem}
 We next turn to the short-interval problem over a subinterval $I$ of $(0,1]$. In this setting, the relevant quantity is the contribution of the squared-spacing sum from Farey fractions in the interval $I$. Denote
\begin{align}\label{def:S0(Q,I)}
    S_0(Q, I) := \sum_{\gamma_j \in F_I(Q)} (\gamma_{j+1} - \gamma_j)^2,
\end{align}
where
\[
F_I(Q) = \left\{ \dfrac{a}{q}  \in I, 1 \leq a \leq q \leq Q, (a, q) = 1 \right\}.
\]
Define
\begin{align}
C_0(Q,I):=\frac{S_0(Q,I)Q^2}{|I|\log Q}.\label{def:C0(Q,I)}
\end{align}
Our next theorem gives an explicit Mundici-type bound after normalizing by $|I|$.

\begin{thm}\label{thm: Mundici Conjecture Short}
  For any integer \(Q\geq \exp\{{e^{22}}\}\) and any subinterval \(I=(\alpha,\beta]\subseteq (0,1]\),
 we have 
  \begin{align}
C_0(Q,I)&<\frac{12}{\pi^2} + \frac{2.01} {\log Q}+\frac{7.06}{|I|\log Q},
\end{align}
where $C_0(Q,I)$ is defined in \eqref{def:C0(Q,I)}. If we further assume \mbox{that $(\beta-\alpha)\geq 4/\log Q$,} then 
\begin{align}
C_0(Q,I)&<3.
\end{align}  
\end{thm}

\begin{rem}
    An upper bound of $C_0(Q,I)$ which holds true for any $Q\geq 2$ is 
    \begin{align}
    C_0(Q,I)\leq \frac{2}{|I|\log 2}.
    \end{align}
    Such a bound follows easily from Theorem~\ref{thm: Mundici conjecture} (which we will show later), but it blows up \mbox{when $|I|$} tends to $0$. However, the bound in Theorem \ref{thm: Mundici Conjecture Short} suggests that, for any fixed interval $I$, when $Q$ is sufficiently large, the upper bound for $C_0(Q,I)$ approaches $12/\pi^2$. Moreover, in the proof of Theorem \ref{thm: Mundici Conjecture Short}, we obtain an explicit asymptotic formula for $S_0(Q,I)$, which indicates that there does not exist an absolute constant $K$, independent of $I$, such that $C_0(Q,I)<K$ for all subintervals $I$.
\end{rem}



\section*{Structure of the paper} 

The paper is organized as follows. In Section \ref{sec: Section 2}, we prove Theorem~\ref{thm: Mundici conjecture} and reduce Theorem~\ref{thm: Mundici conjecture for large Q} to explicit asymptotic estimates for $D_{2,h}(Q)$, $h\geq 2$. We then show that these estimates follow from asymptotic formulas for $S_r(Q)$, $r\geq 1$. Assuming these auxiliary estimates, we complete the proof of Theorem \ref{thm: Mundici conjecture for large Q}. In Section \ref{sec: Section 3}, we reduce Theorem \ref{thm: Mundici Conjecture Short} to an explicit asymptotic formula for $S_0(Q,I)$. Sections \ref{sec: Section 4} and \ref{sec: Section 5} are devoted to the asymptotic formulas required in Section \ref{sec: Section 2} with explicit error bounds. Finally, in Section \ref{sec: Section 6}, we prove the explicit asymptotic formula for $S_0(Q,I)$ needed in Section \ref{sec: Section 3}. 

\section{Auxiliary Theorems for Theorems \ref{thm: Mundici conjecture} and \ref{thm: Mundici conjecture for large Q}}\label{sec: Section 2}
In this section, we give a complete proof of Theorems~\ref{thm: Mundici conjecture} and \ref{thm: Mundici conjecture for large Q}, which concern results \mbox{in $h$-spacings.} Theorem~\ref{thm: Mundici conjecture} requires no additional lemma, whereas the proof of Theorem \ref{thm: Mundici conjecture for large Q} rests on two supporting theorems. 
\subsection{Proof of Theorem~\ref{thm: Mundici conjecture}}

It will be shown later in \eqref{def:D2,h(Q)} that
\[
D_{2,h}(Q) = \sum_{j=1}^N \left(\sum_{r=0}^{h-1}\ell_{j+r}\right)^2. 
\]
By the Cauchy--Schwarz inequality, 
\[
\left(\sum_{r=0}^{h-1}\ell_{j+r}\right)^2\leq h \sum_{r=0}^{h-1} \ell_{j+r}^2. 
\]
Therefore, 
\begin{align*}
    D_{2,h}(Q)\leq h \sum_{r=0}^{h-1}\sum_{j=1}^N\ell_{j+r}^2= h^2 \sum_{j=1}^N \ell_j^2 = h^2D_{2,1}(Q),
\end{align*}
which gives
\[
C_h(Q)\leq h^2C_1(Q). 
\]
By the numerical computation in \cite[Theorem 1.3]{ViAnji}, 
\[
C_1(Q)\leq \frac{2}{\log 2},
\]
and therefore, for all $Q\geq 2$, 
\[
C_h(Q) \leq \frac{2h^2}{\log 2}.
\]
In fact, this maximum is attained at $Q=2$ for all $h\geq 1$. For $h\geq 1$, we have
\begin{align*}
   \gamma_{j+h}-\gamma_j =\frac{h}{2}
\end{align*}
for both $j=1,2$. Therefore,
\[
C_h(2) = \frac{D_{2,h}(2)\times2^2}{\log 2} = \frac{2h^2}{\log 2}.
\]
\subsection{First Supporting Theorem and Proof of Theorem \ref{thm: Mundici conjecture for large Q}}
We now turn to Theorem \ref{thm: Mundici conjecture for large Q}. To prove Theorem \ref{thm: Mundici conjecture for large Q}, it suffices to have the following theorem, which is a version of \cite[Theorem ~1]{BoCoZaha2001} with concrete error bounds.

\begin{thm}\label{thm: Theorem 1}
Let $Q$ be a positive integer, and let $\gamma_1,\gamma_2,\cdots,\gamma_N$ be the $Q$-th Farey sequence. For $h\geq 2$ and $Q\geq \max\{6163, 2^{(h+2)^2}\}$,
\begin{align*}
    \sum_{j=1}^{N} (\gamma_{j+h}-\gamma_j)^2 = \frac{12(2h-1)\log Q}{\pi^2Q^2}+\frac{D(h)}{Q^2}+E_h(Q),
\end{align*}
where $\gamma$ is Euler's constant,  
\[
B= \frac{1}{2}+\log 2 +2\sum_{h=1}^\infty\frac{\zeta(2h)-1}{2h-1}=2.546277\dots,
\]
\begin{align}
    D(h)  =
\frac{12}{\pi^2}
\left[
(2h-1)\left(
\gamma-\frac{\zeta'(2)}{\zeta(2)}
\right)
+\frac{h}{2}+
(h-1)B
+
\sum_{k=2}^{h-1}(h-k)I_k
\right],\label{def:D(h)}
\end{align}
\[
    I_r := \iint\limits_{\mathscr{T}}\frac{dxdy}{xyL_r(x,y)L_{r+1}(x,y)},
    \]
    with $\mathscr{T}$ the Farey triangle defined by $0<x,y\leq 1$ and $x+y>1$ in the plane, $L_r(x,y)$ defined in \eqref{def: Li(x,y) def 1} and \eqref{def: Li(x,y) def 2},
and 
\begin{align*}
    |E_h(Q)|\leq h\cdot 2^{3h+8}
\frac{(\log Q)^{1/(h+2)}}{Q^{2+1/(h+2)}}.
\end{align*}
When $h=2$, the bound in the error term can be improved to
\begin{align*}
    |E_2(Q)| \leq \frac{138\log Q}{Q^{5/2}}+\frac{216(\log Q)^{3/2}}{Q^{11/4}}+\frac{146(\log Q)^2+212\log Q+538}{Q^3}.
\end{align*}
\end{thm}
\begin{rem}
    The analogous result to Theorem~\ref{thm: Theorem 1} for $h=1$ has been obtained in \cite{ViAnji}, and states that
    \begin{align}
    \sum_{j=1}^{N} (\gamma_{j+1}-\gamma_j)^2 &= \frac{12\log Q}{\pi^2 Q^2}-\frac{2}{Q^2}\frac{\zeta'(2)}{\zeta(2)^2} +(2\gamma+1)\frac{6}{Q^2\pi^2}+E_{1},\label{eq:D(2,1)(Q) result}
\end{align}
where 
\[
|E_{1}|\leq \frac{64(\log Q)^2+106\log Q+269}{Q^3}.
\]
\end{rem}
\begin{proof}[Proof of Theorem \ref{thm: Mundici conjecture for large Q}.]
By the definition in \eqref{def:Ch(Q)} and Theorem \ref{thm: Theorem 1},
\begin{align*}
    C_h(Q) =\frac{12(2h-1)}{\pi^2}+\frac{D(h)}{\log Q}+\frac{Q^2E_h(Q)}{\log Q}.
\end{align*}
When $h=1$, by \cite[Theorem 1.3]{ViAnji}, $C_1(Q)<3$ for all $Q>1$. 

For $h=2$, 
\begin{align*}
    |E_2| \leq \frac{138\log Q}{Q^{5/2}}+\frac{216(\log Q)^{3/2}}{Q^{11/4}}+\frac{146(\log Q)^2+212\log Q+538}{Q^3}.
\end{align*}
Define 
\[
G_2(Q):=
\frac{36}{\pi^2}
+
\frac{D(2)}{\log Q}
+
\frac{138}{Q^{1/2}}
+
\frac{216(\log Q)^{1/2}}{Q^{3/4}}
+
\frac{146\log Q+212+\frac{538}{\log Q}}{Q}
\]
Observe that each nonconstant summand in $G_2(Q)$ is decreasing when $Q$ increases, and therefore, it takes the largest value at $Q=6163$, which evaluates to be 
\[
G_2(6163)=7.54749759\dots
\]
Thus, to finish the proof for $h=2$, it suffices to find the smallest $Q_2\geq 6163$ such that  
\[
G_2(Q_2)<6.
\]
 A numerical computation shows that $Q_2=19397$.

For $h\geq 3$, we have
\begin{align*}
    |E_h(Q)|\leq h\cdot 2^{3h+8}
\frac{(\log Q)^{1/(h+2)}}{Q^{2+1/(h+2)}}.
\end{align*}
Therefore, 
\begin{align*}
    C_h(Q)\leq G_h(Q),
\end{align*}
where
\begin{align*}
    G_h(Q)
=
\frac{12(2h-1)}{\pi^2}
+
\frac{D(h)}{\log Q}
+
h2^{3h+8}
\frac{1}
{Q^{1/(h+2)}(\log Q)^{(h+1)/(h+2)}}.
\end{align*}
Observe that $D(h)$ is positive, since each summand in $D(h)$ is positive. In order to complete the proof of the theorem, we need to find a $Q_h\geq\max\{6163,2(h-1),2^{(h+2)^2}\}$ such that
    \begin{align}
    \frac{D(h)}{\log Q_h}
+
\frac{h2^{3h+8}}
{Q_h^{1/(h+2)}(\log Q_h)^{(h+1)/(h+2)}}<3h-\frac{12(2h-1)}{\pi^2}.\label{eq: optimization Qh}
\end{align}
For simplicity, denote 
\begin{align*}
   \Delta_h:=3h-\frac{12(2h-1)}{\pi^2}.
\end{align*}
A sufficient choice of $Q_h$ can be obtained by forcing each term in \eqref{eq: optimization Qh} to be less than $\Delta_h/2$. 
Then we have
\begin{align*}
Q_h&>\exp\left\{
\frac{2D(h)}{\Delta_h}\right\},
\shortintertext{and}
\log Q_h \exp\left\{
\frac{\log Q_h}{h+1}\right\}&>\left(\frac{h\cdot2^{3h+9}}{\Delta_h}\right)^{\frac{h+2}{h+1}},
\end{align*}
which is equivalent to
\begin{align}
Q_h>\exp\left\{(h+1)W\left(\frac{1}{h+1}\left(\frac{h\cdot2^{3h+9}}{\Delta_h}\right)^{\frac{h+2}{h+1}}\right)\right\},\label{eq:upper bound of Qh}
\end{align}
where $W(x)$ is the Lambert $W$ function. Since the lower threshold for $Q_h$ in \eqref{eq:upper bound of Qh} is larger \mbox{than $\max\{6163, 2^{(h+2)^2}\}$} for all $h\geq 3$, we can safely take
\[
Q_h = \exp\left\{\max\left\{
\frac{2D(h)}{\Delta_h},(h+1)W\left(\frac{1}{h+1}\left(\frac{h\cdot2^{3h+9}}{\Delta_h}\right)^{\frac{h+2}{h+1}}\right)\right\}\right\}.
\]
This finishes the proof of Theorem \ref{thm: Mundici conjecture for large Q}.
\end{proof}

\subsection{Second Supporting Theorem} 

In this subsection, we further reduce the proof of Theorem~\ref{thm: Theorem 1}, which is the auxiliary theorem used to prove Theorem~\ref{thm: Mundici conjecture for large Q}, to the following result.
Denote
\begin{align}\label{def:Sr(Q)}
    S_r(Q) = \sum_{j = 1}^{N} (\gamma_{j+1} - \gamma_j)(\gamma_{j+r+1} - \gamma_{j+r})
\end{align}
for each $r \geq 0$.

\begin{thm}\label{thm: Theorem for Sr(Q)}
Let $Q\geq 2, r\geq 1$ be integers, and $S_r(Q)$ be defined as in \eqref{def:Sr(Q)}. 
\mbox{For $Q\geq 6163$,} we have
 \[
S_1(Q) = \frac{6}{\pi^2}Q^{-2} \log Q + AQ^{-2}+R_1, 
\]
where 
\[
A = \frac{6}{\pi^2}\left(\gamma-\frac{\zeta'(2)}{\zeta(2)}+B\right),
\]
the constant $\gamma$ is Euler's constant,  
\begin{align*}
B&= \frac{1}{2}+\log 2 +2\sum_{h=1}^\infty\frac{\zeta(2h)-1}{2h-1}=2.546277\dots,
\shortintertext{and}
|R_1|&\leq \frac{69\log Q}{Q^{5/2}}+\frac{108(\log Q)^{3/2}}{Q^{11/4}}+\frac{9(\log Q)^2}{Q^3}.
\end{align*}
Moreover, for $r\geq 2$ and $Q/\log Q\geq 2^{(r+3)^2}$, we have
 \[
    S_r(Q) = \frac{6I_r}{\pi^2Q^2} + R_r,
    \]
    with 
    \begin{align*}
|R_r|
&\leq
\left(3\cdot 2^{3r+6}
+3\cdot 2^{3r+8}
+\frac{9\cdot 2^{3r+8}r}{\pi^2}\right)
\frac{\log^{1/(r+3)}Q}{Q^{2+1/(r+3)}} +7\cdot 2^{3r+6}\frac{(1+4\log Q)}{(\log Q)^{1/(r+3)}Q^{3-1/(r+3)}}\\
& 
+2^{3r+6}\frac{1}{(\log Q)^{1/(r+3)}Q^{3-1/(r+3)}} +\frac{2^{3r+6}r}{Q^3}
\left(\log Q+2\right),
\end{align*}
    and where \[
    I_r := \iint\limits_{\mathscr{T}}\frac{dxdy}{xyL_r(x,y)L_{r+1}(x,y)},
    \]
    with $\mathscr{T}$ the Farey triangle defined by $0<x,y\leq 1$ and $x+y>1$ in the plane, and $L_r(x,y)$ defined in \eqref{def: Li(x,y) def 1} and \eqref{def: Li(x,y) def 2}.
\end{thm}
\begin{proof}[Proof of Theorem \ref{thm: Theorem 1}.]
In the definition of $S_r(Q)$ in \eqref{def:Sr(Q)}, set $\ell_j=\gamma_{j+1}-\gamma_j$ for $j\geq 0$. It is clear that $\ell_{j+N}=\ell_j = \ell_{N-j-1}$. Then, 
\[
S_r(Q) = \sum_{j=1}^{N}\ell_j\ell_{j+r}.
\]
With this, we may recursively write 
\begin{align}
   D_{2,h}(Q)&=\sum_{j=1}^{N} (\ell_{j+h-1}+\ell_{j+h-2}+\cdots+\ell_j)^2\notag\\
    &=hS_0(Q)+2\sum_{k=1}^{h-1} (h-k)S_k(Q)\label{def:D2,h(Q)}
\end{align}  
for $h\geq 2$. 

Applying \eqref{eq:D(2,1)(Q) result} and Theorem~\ref{thm: Theorem for Sr(Q)} for $r=1$ to \eqref{def:D2,h(Q)}, we obtain for $Q\geq 6163$, 
\begin{align*}
    \sum_{j=1}^{N}(\gamma_{j+2}-\gamma_j)^2 &= 2(S_0(Q)+S_1(Q))\\
    &=\frac{36\log Q}{\pi^2 Q^2} + \frac{12}{Q^2\pi^2}\left(3\gamma-3\frac{\zeta'(2)}{\zeta(2)}+B+1\right)+E_2,
\end{align*}
where
\begin{align*}
    |E_2| \leq \frac{138\log Q}{Q^{5/2}}+\frac{216(\log Q)^{3/2}}{Q^{11/4}}+\frac{146(\log Q)^2+212\log Q+538}{Q^3}.
\end{align*}
This finishes the proof for $h=2$. 

Similarly, for $h\geq 3$, applying \eqref{eq:D(2,1)(Q) result} and Theorem~\ref{thm: Theorem for Sr(Q)} for $r=1,2,3,\cdots,h-1$, we obtain for $Q\geq \max\{6163, 2(h-1)\}$,
\begin{align*}
    \sum_{j=1}^{N} (\gamma_{j+h}-\gamma_j)^2 = \frac{12(2h-1)\log Q}{\pi^2Q^2}+\frac{D(h)}{Q^2}+E_h(Q),
\end{align*}
where 
\begin{align*}
    D(h)  &=
\frac{12}{\pi^2}
\left[
(2h-1)\left(
\gamma-\frac{\zeta'(2)}{\zeta(2)}
\right)
+\frac{h}{2}+
(h-1)B
+
\sum_{k=2}^{h-1}(h-k)I_k
\right]
\shortintertext{and}
    |E_h(Q)|&\leq \frac{138(h-1)\log Q}{Q^{5/2}}
+
\frac{216(h-1)(\log Q)^{3/2}}{Q^{11/4}} 
+
\frac{
(82h-18)(\log Q)^2
}{Q^3} +\frac{106h\log Q}{Q^3}\\
&
\quad+\frac{269h}{Q^3}+
2\sum_{k=2}^{h-1}(h-k)2^{3k+6}
\Bigg[
\left(
15+\frac{36k}{\pi^2}
\right)
\frac{(\log Q)^{1/(k+3)}}{Q^{2+1/(k+3)}}
\\
&\quad+
\frac{28\log Q+8}{(\log Q)^{1/(k+3)}Q^{3-1/(k+3)}} +
\frac{k(\log Q+2)}{Q^3}
\Bigg]\leq h\cdot 2^{3h+8}
\frac{(\log Q)^{1/(h+2)}}{Q^{2+1/(h+2)}}.
\end{align*}
This finishes the proof of Theorem \ref{thm: Theorem 1}.
\end{proof}

Therefore, what remains for the results in the direction of $h$-spacings is to prove Theorem~\ref{thm: Theorem for Sr(Q)}, which we present in Sections \ref{sec: Section 4} and \ref{sec: Section 5}. We split the proof into two parts. In Section \ref{sec: Section 4}, we provide an asymptotic formula for $S_1(Q)$, and leave the formula of $S_r(Q)$ for $r\geq 2$ to Section~\ref{sec: Section 5}.

\section{Auxiliary Theorem for Theorem \ref{thm: Mundici Conjecture Short}}
\label{sec: Section 3}
In this section, we lay out the proof of Theorem \ref{thm: Mundici Conjecture Short}. We show that it suffices to prove the following result.


\begin{thm}\label{thm: asymptotic of S0(Q,I) with error order -2}
   For any integer \(Q>1024\) and any subinterval $I = (\alpha, \beta] \subseteq (0,1]$, we have 
\begin{align}
S_0(Q,I)= \frac{12|I|}{\pi^2} \frac{\log Q}{Q^2} + |I|E_{1,I}
+E_{2,I},\label{eq:asymptotic for S0(Q,I) in Lem 1.8}
\end{align}
where $S_0(Q,I)$ is defined in \eqref{def:S0(Q,I)},
\begin{align*}
    |E_{1,I}|\leq\frac{6(2\gamma+1)}{\pi^2Q^2}-\frac{2\zeta'(2)}{Q^2\zeta(2)^2}+\frac{64(\log Q)^2+106\log Q+269}{Q^3}+\frac{(4Q^{1/10}-2)\log Q}
{(Q^{1/10}-1)^2Q^2}, 
\end{align*}
and
\begin{align*}
|E_{2,I}|
&\leq \frac{\pi^4}{18Q^2}+Q^{-5/2}+4Q^{-21/10+2.1322/(\log\log Q-0.1054)}(1+Q^{-1/10})
(2\log Q+\log^2Q)\notag\\
&\quad+\frac{(4Q^{1/10}-2)\log Q}
{(Q^{1/10}-1)^2Q^2}+\frac{\pi^2Q^{-9/5}}{6(Q^{1/10}-1)^2}.
\end{align*}
\end{thm}

\begin{proof}[Proof of Theorem \ref{thm: Mundici Conjecture Short}.]

Recall that the Mundici-type constant $C_0(Q,I)$ is defined by
\[
C_0(Q,I)=\frac{S_0(Q,I)Q^2}{|I|\log Q}.
\]
Note that by definition, we have
\[
S_0(Q,I)\leq S_0(Q). 
\]
Since the numerical computation in \cite[Theorem 1.3]{ViAnji} gives $S_0(Q)\leq 2/\log 2$, we trivially have 
\[
C_0(Q,I)\leq \frac{2}{|I|\log 2}
\]
for any $Q>1$. To obtain the sharper bound in Theorem \ref{thm: Mundici Conjecture Short}, we proceed as follows.

By the asymptotic formula in \eqref{eq:asymptotic for S0(Q,I) in Lem 1.8}, we obtain that for $Q>1024$,
\begin{align}
C_0(Q,I)\leq \frac{12}{\pi^2} + \frac{E_{1,I}(Q)Q^2} {\log Q}+\frac{E_{2,I}(Q)Q^2}{|I|\log Q}.\label{eq:C_0(Q,I) explicit}
\end{align}
Observe that for \(Q\geq\exp\{{e^{22}}\}\), we have \[\upsilon(Q):=\frac{2.1322}{\log\log Q-0.1054}-\frac{1}{10}<0.\]
Substituting  \(Q\geq\exp\{{e^{22}}\}\) in \eqref{eq:C_0(Q,I) explicit}, we have
\[
E_{1,I}(Q)Q^2<2.01 \quad\text{and}\quad E_{2,I}(Q)Q^2<7.06.
\]

Therefore, we conclude that for $Q\geq \exp\{{e^{22}}\}$, 
\[
C_0(Q,I)< \frac{12}{\pi^2} + \frac{2.01} {\log Q}+\frac{7.06}{|I|\log Q}.
\]
In particular, when $|I|\geq 4/\log Q$, we have 
\[
C_0(Q,I)<3,
\]
which finishes the proof of Theorem \ref{thm: Mundici Conjecture Short}.     
\end{proof}

Thus, what remains in the direction of short intervals is the proof of Theorem \ref{thm: asymptotic of S0(Q,I) with error order -2}, whose proof is postponed until Section \ref{sec: Section 6}.
\section{Asymptotic Formula for \texorpdfstring{$S_1(Q)$}{S1(Q)}}\label{sec: Section 4} 
 
In \cite{Hall1994}, Hall proved that
\[
S_1(Q) = \frac{6}{\pi^2}Q^{-2} \log Q + AQ^{-2}+O\left(\frac{\log Q}{Q^2\sqrt{Q}}\right), 
\]
where 
\[
A = \frac{6}{\pi^2}\left(\gamma-\frac{\zeta'(2)}{\zeta(2)}+B\right),
\]
the constant $\gamma$ is Euler's constant, and 
\[
B= \frac{1}{2}+\log 2 +2\sum_{h=1}^\infty\frac{\zeta(2h)-1}{2h-1}=2.546277\dots
\]
Our goal in this section is to trace Hall's proof in \cite{Hall1994} and obtain the asymptotic formula \mbox{for $S_1(Q)$} in Theorem~\ref{thm: Theorem for Sr(Q)}. The result is as follows.

\begin{thm}\label{thm: S1(Q)}
    Let $Q\geq 6163$ be an integer, and $S_1(Q)$ be defined as in \eqref{def:Sr(Q)}. Then, 
    \[
S_1(Q) = \frac{6}{\pi^2}Q^{-2} \log Q + AQ^{-2}+R_1, 
\]
where 
\[
A = \frac{6}{\pi^2}\left(\gamma-\frac{\zeta'(2)}{\zeta(2)}+B\right),
\]
the constant $\gamma$ is Euler's constant,  
\begin{align*}
B&= \frac{1}{2}+\log 2 +2\sum_{h=1}^\infty\frac{\zeta(2h)-1}{2h-1}=2.546277\dots,
\shortintertext{and}
|R_1|&\leq \frac{69\log Q}{Q^{5/2}}+\frac{108(\log Q)^{3/2}}{Q^{11/4}}+\frac{9(\log Q)^2}{Q^3}.
\end{align*}
\end{thm}
\begin{proof}
    Following \cite[Lemma 2]{HallTenenbaum}, 
    \begin{align}
        S_1(Q) = \sum_{s=1}^Q s^{-2}\sum_{\substack{n=Q-s+1\\(n,s)=1}}^Q\frac{1}{nt},\label{eq:S1(Q) full sum}
    \end{align}
    where $t = t(n,s,Q) := s[(Q+n)/s]-n$. Choose an integer $K\geq 2$. For $2\leq k<K$, \mbox{define $s_{k}= (2Q+1)/k$,} and set $z=(2Q+1)/K$. Now split the sum $S_1(Q)$ into two \mbox{parts, $U_Q$} and $V_Q$, according to whether $s\leq s_K =z$. Hall picked $K=[Q^{1/4}\log^{-1/2}Q]$ at this point. Therefore, keep in mind that $z$ can be computed to only depend on $Q$ from now on. 
    
    We first consider $U_Q$. For $s\leq z$, put $n=Q-n'$, $t=Q-t'$ with $0\leq n',t'\leq s-1$, then 
    \begin{align*}
        \frac{1}{nt} = \frac{1}{Q^2}\frac{1}{(1-n'/Q)(1-t'/Q)}.
    \end{align*}
Upon expansion, we have
\begin{align*}
    \left|\frac{1}{nt}-\frac{1}{Q^2}-\frac{n'+t'}{Q^3}\right| &\leq \frac{s^2}{Q^4}+\frac{2s^2}{Q^4}\frac{1+s/Q}{(1-s/Q)}\\
    &\leq \frac{s^2}{Q^4}+\frac{2s^2}{Q^4}\times 6 = \frac{13s^2}{Q^4}
\end{align*}
when $Q\geq 6163$ (which forces $K\geq 3$). Therefore, substituting back in \eqref{eq:S1(Q) full sum}, 
\begin{align}
    U_Q = \frac{1}{Q^2}\sum_{s\leq z} \frac{\varphi(s)}{s^2}+\frac{1}{Q^3}\sum_{s\leq z} \frac{1}{s^2}\sum_{\substack{Q-s+1\leq n\leq Q\\(n,s)=1}}(n'+t')+R_{1,1},\label{eq:UQ}
\end{align}
    where 
    \begin{align*}
        |R_{1,1}|\leq \sum_{s\leq z}\frac{\varphi(s)}{s^2}\cdot \frac{13s^2}{Q^4}\leq \frac{6z^2}{Q^4}.
    \end{align*}
 The first term in \eqref{eq:UQ} can be written as
  \begin{align*}
      \sum_{s\leq z}\frac{\varphi(s)}{s^2} &= \sum_{s\leq z}\frac{1}{s^2}\sum_{d|s}\mu(d)\frac{s}{d}=\frac{1}{\zeta(2)}\sum_{s\leq z}\frac{1}{s}-\sum_{s\leq z}\frac{1}{s}\sum_{d>z/s}\frac{\mu(d)}{d^2}\\
      &=\frac{6}{\pi^2}(\log z+\gamma)-\sum_{2\leq d\leq z}\frac{\mu(d)}{d^2}\log d+R_{1,2}\\
      &=\frac{6}{\pi^2}(\log z+\gamma)-\frac{\zeta'(2)}{\zeta(2)^2}+R_{1,3},
  \end{align*}
  where 
  \begin{align*}
     |R_{1,2}|&\leq \frac{1}{z}+\frac{1}{z}\sum_{2\leq d\leq z}\frac{1}{d}
\shortintertext{and}
    |R_{1,3}|&\leq \frac{1}{z}+\frac{1}{z}\sum_{2\leq d\leq z}\frac{1}{d}+\frac{4\log z}{z}\leq \frac{6\log z}{z},
  \end{align*}
  when $Q\geq 6163$. Following the discussion in \cite{Hall1994}, the second term in \eqref{eq:UQ} is equivalent to
  \begin{align}
&\frac{1}{Q^3}\sum_{s\leq z} \frac{\varphi(s)}{s}+R_{1,4}\notag\\
=& \frac{z}{Q^3}\sum_{d\leq z}\frac{\mu(d)}{d^2} -\frac{1}{Q^3}\sum_{d\leq z}\frac{\mu(d)}{d}\{z/d\}+R_{1,4}=\frac{6z}{\pi^2Q^3}+R_{1,5},\label{eq:sum of phi(s)/s}
  \end{align}
  where 
  \begin{align*}
  |R_{1,4}|&\leq \frac{1}{Q^3}\sum_{s\leq z} s^{-2}\varphi(s) + \frac{1}{Q^3}\sum_{s\leq z} s^{-1} \tau(s)\leq \frac{\log z+(\log z)^2}{Q^3},
\shortintertext{and}
      |R_{1,5}|
  &\leq \frac{(\log z)^2+2\log z+2}{Q^3}.
  \end{align*}
  
  Combining all, we obtain 
  \begin{align}
      U_Q = \frac{6}{\pi^2Q^2}\left(\log z+\gamma-\frac{\zeta'(2)}{\zeta(2)}+\frac{z}{Q}\right)+R_{1,6},\label{eq:final estimate of U(Q)}
  \end{align}
  where
   \begin{align*}
     |R_{1,6}|\leq \frac{6z^2}{Q^4}+\frac{6\log z}{Q^2z}+ \frac{(\log z)^2+2\log z+2}{Q^3}.
  \end{align*}
  
  We now shift our attention to $V_Q$. Following the original proof, 
  \begin{align}
      V_Q=\sum_{z<s\leq Q}\frac{2}{s^3}\Big\{\frac{1}{k(s)-1}\sum_{\substack{n=Q-s+1\\(n,s)=1}}^Q\frac{1}{n}-\frac{1}{k(s)(k(s)-1)}\sum_{\substack{n=sk(s)-Q\\(n,s)=1}}^Q\frac{1}{n}\Big\},\label{eq:VQ}
  \end{align}
 where $k(s)=[(2Q+1)/s]$ and the right-hand inner sum is empty when $s \mid (2Q+1)$. \mbox{For $u\leq v$} positive integers, 
  \begin{align*}
      \sum_{\substack{n=u}^v} \frac{1}{n} = \frac{\varphi(s)}{s}\log \frac{v}{u}+R_{1,7}, 
  \end{align*}
  where 
  \[
  |R_{1,7}|\leq \frac{2\tau(s)}{u}.
  \]
In \eqref{eq:VQ}, take $u=Q-s+1,v = Q$ for the first inner sum, and take $u=sk(s)-Q,v=Q$ for the second inner sum respectively, we see that the contributions from the total error terms, call it $R_{1,8}$, is bounded by
\begin{align*}
    |R_{1,8}|&\leq \sum_{z<s\leq Q}\frac{2}{s^3}\Big\{\frac{1}{k(s)-1}\frac{2\tau(s)}{Q-s+1}-\frac{1}{k(s)(k(s)-1)}\frac{2\tau(s)}{sk(s)-Q}\Big\}\\
    &\leq \frac{16}{Q^2}\sum_{s>z}\frac{\tau(s)}{s^2}+\frac{32}{Q^3}\sum_{Q/2<s\leq Q}\frac{\tau(s)}{Q-s+1}+4\sum_{z<s\leq Q}\frac{\tau(s)}{s^2Q(Q-s+1)}.
\end{align*}
Using 
\begin{align*}
    \sum_{s>z}\frac{\tau(s)}{s^2} &= \sum_{d=1}^\infty\sum_{\substack{s>z\\d|s}}\frac{1}{s^2}\leq \frac{\log z+3}{z}
\shortintertext{and}   \sum_{Q/2<s\leq Q}\frac{\tau(s)}{Q-s+1}&\leq 2\sqrt{Q}\sum_{1\leq s\leq Q/2}\frac{1}{s}\leq  2\sqrt{Q}(1+\log Q),
 \end{align*}
 where the first inequality comes from trivially bounding $\tau(s)$ by $2\sqrt{s}$, we arrive at
\begin{align*}
     |R_{1,8}|\leq \frac{20(\log z+3)}{Q^2z}+\frac{64(1+\log Q)}{Q^{5/2}}.
\end{align*}
Therefore, we are left to approximate the main term in \eqref{eq:VQ}, which is
\begin{align}
    M_{V,Q}:=\sum_{z<s\leq Q}\frac{2}{s^4}\Big\{\frac{\varphi(s)}{k(s)-1}\log\frac{Q}{Q-s+1}-\frac{\varphi(s)}{k(s)(k(s)-1)}\log\frac{Q}{sk(s)-Q}\Big\}.\label{eq:M(V,Q)}
\end{align}

Split the sum into ranges $(s_{k+1},s_k]$ for $2\leq k<K$, so $k(s)=k$. Using \eqref{eq:sum of phi(s)/s} and partial summation, we obtain
\begin{align}
    \sum_{s_{k+1}<s\leq s_k}\frac{\varphi(s)}{s^4}\log\frac{Q}{Q-s+1} = \frac{6}{\pi^2}\int_{s_{k+1}}^{s_k}s^{-3}\log\frac{Q}{Q-s+1}\ ds+R_{1,9},\label{eq:first summand in M(V,Q)}
\end{align}
where
\begin{align*}
|R_{1,9}|&\leq 2(\log s_k+2)s_{k+1}^{-3}\log\frac{Q}{Q-s_{k}+1}\\
&\quad+\int_{s_{k+1}}^{s_k}3(\log s+2)s^{-4}\log\frac{Q}{Q-s+1}+(\log s+2)\frac{s^{-3}}{Q-s+1}\ ds\\
&\leq \frac{62(\log Q)^2}{Qs_k^2}
\end{align*}
by bounding $\log Q/(Q-s+1)$ by $3s/Q$ for $k\geq 3$ and by $\log (Q+1/2)$ when $k=2$. Similarly, 
\begin{align}
     \sum_{s_{k+1}<s\leq s_k}\frac{\varphi(s)}{s^4}\log\frac{Q}{sk-Q} = \frac{6}{\pi^2}\int_{s_{k+1}}^{s_k}s^{-3}\log\frac{Q}{sk-Q}\ ds+R_{1,10},\label{eq:second summand in M(V,Q)}
\end{align}
where
\begin{align*}
   |R_{1,10}| \leq \frac{16k(\log Q)^2}{Qs_k^2}.
\end{align*}
Substituting \eqref{eq:first summand in M(V,Q)} and \eqref{eq:second summand in M(V,Q)} in \eqref{eq:M(V,Q)} together with taking $R_{1,8}$ into account, we arrive at
\begin{align}
    V_Q = \frac{12}{\pi^2}\int_{z}^{Q+1/2}\Big\{\frac{1}{(k(s)-1)}\log\frac{Q}{Q-s+1}-\log\frac{Q}{sk(s)-Q}\Big\}s^{-3}\ ds+R_{1,11},\label{eq:VQ main term + error term}
\end{align}
where 
\begin{align*}
    |R_{1,11}|&\leq \sum_{k=2}^{K-1}\frac{124(\log Q)^2}{(k-1)Qs_k^2}+\frac{32k(\log Q)^2}{k(k-1)Qs_k^2}\\
    &\leq 39K^2Q^{-3}(\log Q)^2.
\end{align*}

Denote the main term in \eqref{eq:VQ main term + error term} by $I_Q$. Upon substituting $s=(2Q+1)/x$, we have
\begin{align*}
    (2Q+1)^2I_Q &= \frac{12}{\pi^2}\int_2^K\Big\{\frac{x}{[x]-1}\log\frac{Qx}{(Q+1)x-2Q-1}\\
    &\quad+\frac{x}{[x]([x]-1)}\log\frac{(2Q+1)[x]-Qx}{Qx}\Big\} \ dx.
\end{align*}

We have 
\begin{align*}
    \log\frac{Qx}{(Q+1)x-2Q-1} = \log\frac{x}{x-2}-\log \left(1+\frac{x-1}{Q(x-2)}\right) = \log\frac{x}{x-2} + R_{1,12}
\end{align*}
for $x\geq 3$, where 
\begin{align*}
|R_{1,12}|&\leq \log (1+2/Q)\leq 2/Q,
\shortintertext{and}
\log\frac{(2Q+1)[x]-Qx}{Qx} &= \log \left(2\frac{[x]}{x}-1\right)+R_{1,13}
\end{align*}
for $x\geq 2$, where
\[
|R_{1,13}|\leq 2/Q.
\]
Therefore, 
\begin{align*}
    I_Q = \frac{12}{\pi^2(2Q+1)^2}\int_2^K \left(f(x)+\frac{2}{x}\right)\ dx + R_{1,14},
\end{align*}
where 
\[
f(x):=-\frac{2}{x}+\frac{x}{[x]-1}\log\frac{x}{x-2}+\frac{x}{[x]([x]-1)}\log \left(2\frac{[x]}{x}-1\right),
\]
and 
\begin{align*}
    |R_{1,14}|&\leq \frac{12}{\pi^2(2Q+1)^2} \int_3^K\left\{\frac{x}{[x]-1}\frac{2}{Q}+\frac{x}{[x]([x]-1)}\frac{2}{Q}\right\}\ dx\\
    &\leq \frac{18(K-3+\log Q)}{\pi^2Q^3}+\frac{27+18\log Q}{\pi^2Q^3},
\end{align*}
where the last summand in the last line comes from approximating the integral by $f(x)+2/x$ over $2\leq x\leq 3$. 
Putting everything together, we obtain
\begin{align}
    V_Q = \frac{24}{\pi^2(2Q+1)^2}\left(\log\frac{K}{2}+B-B(K) \right)+R_{1,15},\label{eq:V(Q) final estimate}
\end{align}
where
\begin{align*}
    B = \frac{1}{2}\int_2^\infty f(x)dx, \quad B(K) = \frac{1}{2}\int_K^\infty f(x)\ dx,
\end{align*}
and
\begin{align*}
|R_{1,15}|&\leq 39K^2Q^{-3}(\log Q)^2+\frac{18K-27+36\log Q}{\pi^2Q^3}+\frac{24\log K}{\pi^2(2Q+1)^2}\\
&\leq \frac{39\log Q}{Q^{5/2}}+\frac{18}{\pi^2Q^{11/4}(\log Q)^{1/2}}+\frac{36\log Q-26}{\pi^2Q^3}.
\end{align*}
The constant $B$ was computed in \cite{Hall1994} exactly. To estimate $B(K)$, a calculation shows that \mbox{if $x\geq 3$,} then 
\[
f(x)=\frac{4}{x^2}+\left(\frac{20}{3}+4\theta(1-\theta)\right)\frac{1}{x^3}+R_{1,16},
\]
where $|R_{1,16}|\leq 1000x^{-4}$. Therefore,
\begin{align}
    B(K) &= \frac{1}{2}\int_K^\infty\frac{4}{x^2}\ dx+\frac{1}{2}\int_K^\infty\left(\frac{20}{3}+4\theta(1-\theta)\right)x^{-3}\ dx+R_{1,17}\notag\\
    &=\frac{2}{K} +\frac{11}{6K^2}-2\int_K^\infty B_2(\{x\})x^{-3}\ dx+R_{1,17}\notag\\
    & = \frac{2}{K} +\frac{11}{6K^2}+R_{1,18},\label{eq:estimate of B(K)}
\end{align}
where 
\begin{align*}
|R_{1,17}|&\leq \frac{1}{2}\int_K^\infty \frac{1000}{x^4}\ dx = \frac{500}{3K^3}, 
\shortintertext{and} 
|R_{1,18}|&\leq \frac{500}{3K^3}+\frac{\sqrt{3}}{54K^3}\leq \frac{167}{K^3}.
\end{align*}
Substituting \eqref{eq:estimate of B(K)} back into the estimate for $V(Q)$ in \eqref{eq:V(Q) final estimate}, and combining with the asymptotic formulas for $U_Q$ in \eqref{eq:final estimate of U(Q)}, we finally arrive at
\begin{align*}
    S_1(Q) &= \frac{6}{\pi^2Q^2}\left(\log z+\gamma-\frac{\zeta'(2)}{\zeta(2)}+\frac{z}{Q}\right)+\frac{24}{\pi^2(2Q+1)^2}\left(\log\frac{K}{2}+B-\frac{2}{K}-\frac{11}{6K^2} \right)\\
    &\quad+R_{1,19},
\end{align*}
where 
\begin{align*}
    R_{1,19}&\leq \frac{6z^2}{Q^4}+\frac{6\log z}{Q^2z}+ \frac{(\log z)^2+2\log z+2}{Q^3}\\
    &\quad+ \frac{39\log Q}{Q^{5/2}}+\frac{18}{\pi^2Q^{11/4}(\log Q)^{1/2}}+\frac{36\log Q-26}{\pi^2Q^3}+\frac{24}{\pi^2(2Q+1)^2}\cdot\frac{167}{K^3}.
\end{align*}
Substituting the values of $z$ and $K$, we will obtain Theorem \ref{thm: S1(Q)}.
\end{proof}


\section{An Asymptotic formula for \texorpdfstring{$S_r(Q)$}{Sr(Q)}}\label{sec: Section 5}
In this section, our goal is to prove the asymptotic formula of $S_r(Q)$ for $r\geq 2$ presented in Theorem~\ref{thm: Theorem for Sr(Q)}. The result is as follows.

\begin{thm}\label{thm: asymptotic for Sr(Q)}
    Let $r\geq 2$. Define
    \[
    I_r := \iint\limits_{\mathscr{T}}\frac{dxdy}{xyL_r(x,y)L_{r+1}(x,y)},
    \]
    where $\mathscr{T}$ denotes the Farey triangle defined by $0<x,y\leq 1$ and $x+y>1$ in the plane, \mbox{and $L_r(x,y)$} is defined in \eqref{def: Li(x,y) def 1} and \eqref{def: Li(x,y) def 2}. Then, for $Q$ such that $Q/\log Q\geq 2^{(r+3)^2}$, we have
    \[
    S_r(Q) = \frac{6I_r}{\pi^2Q^2} + R_r,
    \]
    where 
    \begin{align*}
|R_r|
&\leq
\left(3\cdot 2^{3r+6}
+3\cdot 2^{3r+8}
+\frac{9\cdot 2^{3r+8}r}{\pi^2}\right)
\frac{\log^{1/(r+3)}Q}{Q^{2+1/(r+3)}} +7\cdot 2^{3r+6}\frac{(1+4\log Q)}{(\log Q)^{1/(r+3)}Q^{3-1/(r+3)}}\\
& 
+2^{3r+6}\frac{1}{(\log Q)^{1/(r+3)}Q^{3-1/(r+3)}} +\frac{2^{3r+6}r}{Q^3}
\left(\log Q+2\right).
\end{align*}
\end{thm}

We follow the proof structure in Sections~3--7 of \cite{BoCoZaha2001}, while obtaining concrete bounds for the error terms in all necessary lemmas.

\subsection{Notation in \cite{BoCoZaha2001}}

We follow the same notation in \cite{BoCoZaha2001}. We introduce the functions $L_i$ defined on $\mathscr{T} = \{(x,y): 0 < x, y \leq 1, x+y>1\}$ by
\begin{align}
   L_0(x,y) = x, \quad L_1(x,y) = y\label{def: Li(x,y) def 1} 
\end{align}
and
\begin{align}\label{def: Li(x,y) def 2}
    L_i(x,y) = \begin{cases} 
    \left[ \dfrac{1+L_{i-2}(x,y)}{L_{i-1}(x,y)}\right]L_{i-1}(x,y)-L_{i-2}(x,y)  & \text{ if } i\geq 2
       \\
    \left[ \dfrac{1+L_{i+2}(x,y)}{L_{i+1}(x,y)}\right]L_{i+1}(x,y)-L_{i+2}(x,y)  & \text{ if } i\leq -1.
   \end{cases}
\end{align}
Additionally, consider the function
\begin{align}
    f_r(x,y) = \dfrac{1}{xyL_r(x,y)L_{r+1}(x,y)}.\label{def:fr(x,y)}
\end{align}
    
Now, for a fixed $\mathbf{k}\in\N^r$, we define
\begin{align*}
    L_{\mathbf{k}, 0}(x,y) = x, L_{\mathbf{k}, 1}(x,y) = y, 
\end{align*}
then recursively, for $i \in \{ 2, \dots, r+1\}$, the linear function
\begin{align}\label{def: L[k,i](x,y)}
    L_{\mathbf{k},i}(x,y)= k_{i-1}L_{\mathbf{k}, i-1}(x,y) - L_{\mathbf{k}, i-2}(x,y), \quad (x,y) \in \mathbb{R}^2.
\end{align}
We also want to consider the set of indices
\begin{align*}
    \mathscr{L}_{\mathbf{k}} := \{ 1, \dots, N(Q)\} \cap \{j;q_{j+1} = L_{\mathbf{k},i}(q_j, q_{j+1}) \text{ for all } i \in \{2, \dots, r+1\}\}.
\end{align*}
In other words,
    \[
    \mathscr{L}_{\mathbf{k}} =\{\text{index $j$}: q_{j+i+1} = k_iq_{j+i}-q_{j+i-1}\},
    \]
    where $k_i = [\frac{Q+q_{i-1}}{q_i}]$.

For each $r \geq 0$ and $\mathbf{k} \in (\N^{*})^r$,
\begin{align*}
    S_{r,\mathbf{k}} &= \sum_{j \in \mathscr{L}_\mathbf{k}} (\gamma_{j+1} - \gamma_j)(\gamma_{j+r+1} - \gamma_{j+r}) = \sum_{j \in \mathscr{L}_\mathbf{k}} \dfrac{1}{q_jq_{j+1}}\cdot\dfrac{1}{q_{j+r}q_{j+r+1}}\\
    &= \sum_{j \in \mathscr{L}_\mathbf{k}} \dfrac{1}{q_jq_{j+1}L_{\mathbf{k},r}(q_j,q_{j+1})L_{\mathbf{k},r+1}(q_j,q_{j+1})}.
\end{align*}
Then, we have
\begin{align*}
    S_r(Q) = \sum_{\mathbf{k} \in (\N^*)^r} S_{r,\mathbf{k}}(Q).
\end{align*}
These sums can be truncated to
\begin{align*}
    S_{r,T}(Q) = \sum_{1\leq k_1, k_2, \dots, k_r \leq T} S_{r,\mathbf{k}}(Q),
\end{align*}
where $T \geq 1$. Also, define
\begin{align*}
    f_{r, \mathbf{k}}(x,y) = \dfrac{1}{xyL_{\mathbf{k},r}(x,y)L_{\mathbf{k},r+1}(x,y)}. 
\end{align*}

\subsection{Preliminary lemmas}
Let $\Omega\subseteq \R^2$ be a convex bounded region with rectifiable boundary~$\partial\Omega$ and assume that $f$ is a $C^1$ function on $\Omega$. Denote $\|f\|_\infty =  \sup_{(x,y)\in\Omega} |f(x,y)|$ and set
\[
S = S(f,\Omega) = \sum_{(a,b)\in\Omega\cap\Z^2} f(a,b).
\]

\begin{lem}\textbf{(Effective version of Lemma 1 in \cite{BoCoZaha2001})}\label{lem: Lemma 1}
    Suppose that $\Omega$ and $f$ are as above. Then
    \begin{align*}
    \left|S-\iint\limits_\Omega f(x,y)dxdy\right|\leq \left(\Big\|\frac{\partial f}{\partial x}\Big\|_\infty+\Big\|\frac{\partial f}{\partial y}\Big\|_\infty\right)\text{Area}(\Omega)+7\|f\|_\infty (1+\text{length}(\partial\Omega)).
    \end{align*}
\end{lem}
\begin{proof}
    Denote $R_{a,b}=[a,a+1]\times[b,b+1]$ for $a,b\in\Z$. We have 
    \[
     \left|S-\iint\limits_\Omega f(x,y)dxdy\right|\leq \|f\|_\infty E(\Omega),
    \]
    where
    \[
    E(\Omega) := \sum_{\substack{a,b\in\Z\\ R_{a,b}\cap\partial\Omega\neq \emptyset}} 1 = \text{Area}(\bigcup_{\substack{a,b\in\Z\\ R_{a,b}\cap\partial\Omega\neq \emptyset}} R_{a,b}).
    \]
  For any unit square $R_{a,b}$ to overlap with $\partial\Omega$, we will have 
  \[
  \text{dist}(z,\partial\Omega)\leq \sqrt{1^2+1^2}=\sqrt{2}.
  \]
  Therefore, by the Steiner Formula (for example, see \cite{Morvan}), we have
  \begin{align*}
      E(\Omega)\leq 2\times\sqrt{2}\times\text{length}(\partial\Omega)+\pi(\sqrt{2})^2.
  \end{align*}
  Combining with the rest of the proof in \cite[Lemma 1]{BoCoZaha2001}, we obtain Lemma \ref{lem: Lemma 1}. 
\end{proof}

Define 
\begin{align}
S' = S'(f,\Omega) = \sum_{\substack{(a,b)\in\Omega\cap\Z^2\\\gcd(a,b)=1}} f(a,b).\label{def:S'}
\end{align}

\begin{lem}\textbf{(Effective version of Lemma 2 and Corollary 2 in \cite{BoCoZaha2001})}\label{lem: Lemma 2}
    Suppose that $\Omega$ and $f$ are as above. In addition, suppose $\Omega\subseteq [1,R]\times [1,R]$. Then
    \begin{align*}
    \left|S'-\frac{6}{\pi^2}\iint\limits_\Omega f(x,y)dxdy\right|&\leq \left(\Big\|\frac{\partial f}{\partial x}\Big\|_\infty+\Big\|\frac{\partial f}{\partial y}\Big\|_\infty\right)\text{Area}(\Omega)\log R\\
    &\quad+7\|f\|_\infty (R+4R\log R)+\|f\|_\infty R.
    \end{align*}
\end{lem}
\begin{proof}
    Following the proof of \cite[Lemma 2]{BoCoZaha2001} with Lemma \ref{lem: Lemma 1}, we end up with 
     \begin{align}
    \left|S'-\sum_{d=1}^R\frac{\mu(d)}{d^2}\iint\limits_\Omega f(x,y)dxdy\right|&\leq \left(\Big\|\frac{\partial f}{\partial x}\Big\|_\infty+\Big\|\frac{\partial f}{\partial y}\Big\|_\infty\right)\text{Area}(\Omega)\log R\notag\\
    &\quad+7\|f\|_\infty (R+\text{length}(\partial\Omega)\log R).\label{eq:estimation of S'}
    \end{align}

    Now 
    \begin{align}
      \sum_{d=1}^R\frac{\mu(d)}{d^2}  = \sum_{d=1}^\infty\frac{\mu(d)}{d^2} + E_1(R), \label{eq:approximation of zeta(2)} 
    \end{align}
    where 
    \begin{align*}
    |E_1(R)|\leq \sum_{d=R+1}^\infty \frac{1}{d^2}\leq \int_R^\infty \frac{1}{t^2}dt\leq \frac{1}{R}.
    \end{align*}
Note that the main term \eqref{eq:approximation of zeta(2)} is $1/\zeta(2) =6/\pi^2$.  Combining \eqref{eq:estimation of S'} and \eqref{eq:approximation of zeta(2)}, we arrive at
     \begin{align*}
    \left|S'-\frac{6}{\pi^2}\sum_{d=1}^R\iint\limits_\Omega f(x,y)dxdy\right|&\leq \left(\Big\|\frac{\partial f}{\partial x}\Big\|_\infty+\Big\|\frac{\partial f}{\partial y}\Big\|_\infty\right)\text{Area}(\Omega)\log R\notag\\
    &\quad+7\|f\|_\infty (R+\text{length}(\partial\Omega)\log R)+\frac{1}{R}\iint\limits_\Omega f(x,y)dxdy.
    \end{align*}
    Observe that
    \begin{align*}
        \frac{1}{R}\iint\limits_\Omega f(x,y)dxdy\leq \frac{\|f\|_\infty \text{Area}(\Omega) }{R}\leq \|f\|_\infty R. 
    \end{align*}
    Moreover, since $\Omega$ is convex, the projection of $\partial\Omega$ to $x$ and $y$-axis is of length at most $R$. Therefore, 
    \[
    \text{length}(\partial\Omega)\leq 4R.
    \]
    This finishes the proof of Lemma \ref{lem: Lemma 2}.
\end{proof}

\begin{coro}\textbf{(Effective version of Lemma 5 in \cite{BoCoZaha2001})}\label{coro: Corollary 2} Suppose that $r\geq 2$, $\mathbf{k}\in\N^r, M\geq 1$ and $(x,y)\in\Omega_{\mathbf{k},M}$. Then
\[
|f_{r,\mathbf{k}}(x,y)|\leq 2^{3r+6}\frac{M}{Q^4}.
\]
    
\end{coro}
\begin{proof}
        By definition of $\Omega_{\mathbf{k},M}$, we have $x,y,L_{\mathbf{k},r}(x,y)$, and $L_{\mathbf{k},r+1}(x,y)$ all $\geq Q/M$. By \cite[Lemma 5]{BoCoZaha2001}, we have at least three of these four expressions $\geq 2^{-r-2}Q$. Therefore,
    \[
    xyL_{\mathbf{k},r}(x,y)L_{\mathbf{k},r+1}(x,y)\geq \frac{Q}{M}(2^{-r-2}Q)^3 = 2^{-3r-6}\frac{Q^4}{M}.
    \]
    Therefore, 
    \[
    |f_{r,\mathbf{k}}(x,y)| = \frac{1}{xyL_{\mathbf{k},r}(x,y)L_{\mathbf{k},r+1}(x,y)}\leq 2^{3r+6}\frac{M}{Q^4}.
    \]
    This finishes the proof.
\end{proof}
\begin{coro}\textbf{(Effective version of Corollary 3 in \cite{BoCoZaha2001})}\label{coro: Corollary 3}
    Suppose that $r\geq 2$, $\mathbf{k}\in\N^r, M\geq 1$ and $(x,y)\in\Omega_{\mathbf{k},M}$. Then
    \[
    \max \left(\Big|\frac{\partial f_{r,\mathbf{k}}}{\partial x}(x,y)\Big|, \Big|\frac{\partial f_{r,\mathbf{k}}}{\partial y}(x,y)\Big|\right)\leq 
3\cdot 2^{3r+6}\frac{k_1k_2\cdots k_{r}M^2}{Q^5}.    \]
\end{coro}
\begin{proof}
Denote $g = f_{r,\mathbf{k}}^{-1}$. Then
\begin{align*}
    \frac{\partial f_{r,\mathbf{k}}}{\partial x} &= -g^{-2}\frac{\partial g}{\partial x} =  -g^{-2}y(L_{\mathbf{k},r}L_{\mathbf{k},r+1}+x(L_{\mathbf{k},r}\cdot\partial_x L_{\mathbf{k},r+1}+\partial_x L_{\mathbf{k},r}\cdot L_{\mathbf{k},r+1}))\\
    &=-f_{r,\mathbf{k}}\left(\frac{1}{x}+\frac{\partial_xL_{\mathbf{k},r}}{L_{\mathbf{k},r}}+\frac{\partial_x L_{\mathbf{k},r+1}}{L_{\mathbf{k},r+1}}\right).
\end{align*}
Similarly, 
\[
  \frac{\partial f_{r,\mathbf{k}}}{\partial y} = -f_{r,\mathbf{k}}\left(\frac{1}{y}+\frac{\partial_yL_{\mathbf{k},r}}{L_{\mathbf{k},r}}+\frac{\partial_y L_{\mathbf{k},r+1}}{L_{\mathbf{k},r+1}}\right).
\]
By Corollary \ref{coro: Corollary 2}, 
\[
|f_{r,\mathbf{k}}|\leq 2^{3r+6}\frac{M}{Q^4}.
\]

Moreover, by the recursive definition of $L_{\mathbf{k},i}(x,y)$ in \eqref{def: L[k,i](x,y)}, 
\[
\partial_xL_{\mathbf{k},r},\partial_yL_{\mathbf{k},r}\leq k_1k_2\cdots k_{r-1}.
\]
Combining with the property of $\Omega_{\mathbf{k},M}$, we obtain
\begin{align*}
 \left|\frac{\partial f_{r,\mathbf{k}}}{\partial x} \right|&\leq 2^{3r+6}\frac{M}{Q^4}\left(\frac{M}{Q}+\frac{k_1k_2\cdots k_{r-1}M}{Q}+\frac{k_1k_2\cdots k_{r}M}{Q}\right)\\
 &\leq 3\cdot 2^{3r+6}\frac{k_1k_2\cdots k_{r}M^2}{Q^5}.
\end{align*}
Same bound works for $\partial_y f_{r,\mathbf{k}}$. This finishes the proof of the lemma. 
\end{proof}
Let 
\begin{align*}
    \Omega_k = \{ (x,y) \in \R^2 : &0 <L_{\bk,i}(x,y)\leq Q \text{ for all } 0 \leq i \leq r+1,\\
    &Q <L_{\bk,i}(x,y) + L_{\bk,i+1}(x,y) \text{ for all } 0 \leq i \leq r\}.
\end{align*}
For each $M \geq 1$ and $\mathbf{k} \in (\N^*)^r$, we consider its convex subset
\begin{align*}
   \Omega_{\mathbf{k},M} &= \{ (x,y) \in \Omega_\mathbf{k}: \min\left(x,y,L_{\bk,r}(x,y), L_{\bk,r+1}(x,y)\right) \geq Q/M \}\\
   &= \Omega_\mathbf{k} \cap [Q/M, \infty)^2 \cap \bigcap_{i \in \{r, r+1\}} \{(x,y) \in \R^2: L_{\bk,i}(x,y) \geq Q/M\}.   
\end{align*}

Also, let 
\begin{align*}
    \fanm_{\bk} = \{(a,b) \in \Omega_\bk \cap \Z^2: \gcd(a,b) = 1\}.
\end{align*}
We then consider its subset 
\begin{align*}
    \fanm_{\bk, M} = \fanm \cap \Omega_{\bk,M}
\end{align*}
of $\fanm_\bk$ and the sum
\begin{align}
    S_{r,\bk,M}(Q) = \sum_{(a,b) \in \fanm_{\bk,M}} f_{r, \bk}(a,b).\label{def:S(r,k,M)}
\end{align}

\begin{lem}\textbf{(Effective version of Lemma 6 in \cite{BoCoZaha2001})}\label{lem: Lemma 6}
    Suppose that $r\geq 2$ and $M\geq 1$. Then
    \[
    \sum_{\mathbf{k}\in \N^r} |S_{r,\mathbf{k}}(Q)-S_{r,\mathbf{k},M}(Q)|
  \leq \frac{2^{3r+8}}{MQ^2}.  \]
\end{lem}
\begin{proof}
   By \cite[Lemma 5]{BoCoZaha2001}, among $q_j, q_{j+1},q_{j+r}, q_{j+r+1}$, we have at least three of these four numbers $\geq 2^{-r-2}Q$. Therefore, following the original proof, 
   \begin{align*}
       \sum_{\mathbf{k}\in \N^r} |S_{r,\mathbf{k}}(Q)-S_{r,\mathbf{k},M}(Q)|&\leq \sum_{\substack{1\leq j\leq N\\\min(q_j,q_{j+1}, q_{j+r}, q_{j+r+1})<Q/M}}\frac{2^{3r+6}}{Q^3\min(q_j,q_{j+1}, q_{j+r}, q_{j+r+1})}\\
       &\leq \frac{2^{3r+8}}{MQ^2}.
   \end{align*}
\end{proof}

Now we are ready to approximate $S_{r}(Q)$ using Lemma \ref{lem: Lemma 2}. For $T,M
\in\N$ and $k\in\N^r$ \mbox{with $r\geq 2$,} define
\begin{align}
    \mathscr{D}(T,M) = \frac{1}{Q}\bigcup_{1\leq k_1,\cdots, k_r\leq T}\Omega_{\mathbf{k},M}.\label{def: D(T,M)}
\end{align}
 An important remark is that $\mathscr{D}(T,M)$ is indeed independent of $Q$. To that end, we have the following lemma.
 
\begin{lem}\label{lem: approximate Q squared Sr(Q)}
    Suppose $r\geq 2$, $M\geq 1$ and $2Q\geq T\geq 2^{r+3}$. Then,
    \begin{align*}
    |Q^2S_r(Q)-\frac{6}{\pi^2}\iint\limits_{\mathscr{D}(T,M)}f_r(x,y)\ dxdy|&\leq 3\cdot 2^{3r+6}\frac{T^r M^2 \log Q}{Q}+7\cdot 2^{3r+6}\frac{M(1+4\log Q)}{Q}\notag\\
     &\quad+ 2^{3r+6}\frac{M}{Q}+\frac{2^{3r+8}}{M}+\frac{2^{3r+6} r}{Q}\left(\frac{12Q}{\pi^2 T}+\log\frac{2Q}{T}+2\right).
    \end{align*}
\end{lem}
\begin{proof}
Apply Lemma \ref{lem: Lemma 2} with $R=Q, \ \Omega =\Omega_{\mathbf{k},M}$ and $f=f_{r,\mathbf{k}}$, we obtain
\begin{align*}
    \Big|S_{r,\mathbf{k},M}(Q)-\frac{6}{\pi^2}\iint\limits_{\Omega_{\mathbf{k},M}}f_{r,\mathbf{k}}(x,y)\ dxdy\Big|&\leq \left(\Big\|\frac{\partial f_{r,\mathbf{k}}}{\partial x}\Big\|_\infty+\Big\|\frac{\partial f}{\partial y}\Big\|_\infty\right)\text{Area}(\Omega_{\mathbf{k},M})\log Q\\
    &\quad+7\|f_{r,\mathbf{k}}\|_\infty (Q+4Q\log Q)+\|f_{r,\mathbf{k}}\|_\infty Q.
\end{align*}
where $S_{r,\mathbf{k},M}$ is defined in \eqref{def:S(r,k,M)}. Applying Corollaries \ref{coro: Corollary 2} and \ref{coro: Corollary 3}, and bounding $\text{Area}(\Omega_{\mathbf{k},M})$ by $\text{Area}(\Omega_{\mathbf{k}})$, we obtain
\begin{align*}
    \Big|S_{r,\mathbf{k},M}(Q)-\frac{6}{\pi^2}\iint\limits_{\Omega_{\mathbf{k},M}}f_{r,\mathbf{k}}(x,y)\ dxdy\Big|&\leq 6\cdot 2^{3r+6}\frac{k_1k_2\cdots k_r M^2\text{Area}(\Omega_{\mathbf{k}})\log Q}{Q^5}\\
    &\quad+7\cdot 2^{3r+6}\frac{M(1+4\log Q)}{Q^3}+ 2^{3r+6}\frac{M}{Q^3}.
\end{align*}
Therefore,
\begin{align}
    &\Big|\sum_{1\leq k_1,\cdots,k_r\leq T}S_{r,\mathbf{k},M}(Q)-\frac{6}{\pi^2}\sum_{1\leq k_1,\cdots,k_r\leq T}\iint\limits_{\Omega_{\mathbf{k},M}}f_{r,\mathbf{k}}(x,y)\ dxdy\Big|\label{eq:approximate S(r,k,M) with double integral}\\
    \leq&\ 3\cdot 2^{3r+6}\frac{T^r M^2 \log Q}{Q^3}+7\cdot 2^{3r+6}\frac{M(1+4\log Q)}{Q^3}+ 2^{3r+6}\frac{M}{Q^3}.\notag
\end{align}
Recall the definition of $f_{r}(x,y)$ and $\mathscr{D}(T,M)$ in \eqref{def:fr(x,y)} and \eqref{def: D(T,M)} The second summand in \eqref{eq:approximate S(r,k,M) with double integral} can be reformulated as
\[
\frac{6}{\pi^2}\sum_{1\leq k_1,\cdots,k_r\leq T}\iint\limits_{\Omega_{\mathbf{k},M}}f_{r,\mathbf{k}}(x,y)\ dxdy = \frac{6}{\pi^2}\iint\limits_{\mathscr{D}(T,M)}f_{r}(x,y)\ dxdy.
\]
Applying Lemma \ref{lem: Lemma 6} to the first summand in \eqref{eq:approximate S(r,k,M) with double integral}, we arrive at
\begin{align*}
   \Big|S_{r,T}(Q)-\sum_{1\leq k_1,\cdots,k_r\leq T}S_{r,\mathbf{k},M}(Q)\Big|\leq \frac{2^{3r+8}}{MQ^2}.
\end{align*}

Combining all, we get 
\begin{align}
     \Big|Q^2S_{r,T}(Q)-\frac{6}{\pi^2}\iint\limits_{\mathscr{D}(T,M)}f_{r}(x,y)\ dxdy\Big|&\leq 3\cdot 2^{3r+6}\frac{T^r M^2 \log Q}{Q}+7\cdot 2^{3r+6}\frac{M(1+4\log Q)}{Q}\notag\\
     &\quad+ 2^{3r+6}\frac{M}{Q}+\frac{2^{3r+8}}{M}.\label{eq:approximate SrT(Q) by double integral}
\end{align}
Finally, we approximate $S_r(Q)$ by $S_{r,T}(Q)$. Similar to the proof of Lemma \ref{lem: Lemma 6}, following the proof of \cite[Lemma 7]{BoCoZaha2001}, we have
\begin{align*}
   \Big|S_{r,T}(Q)-S_r(Q) \Big|&\leq \frac{2^{3r+6}}{Q^3}\sum_{j_0=1}^r\sum_{\mathbf{k}, k_{j_0}>T}\sum_{j\in\mathscr{L}_\mathbf{k}}\frac{1}{\min(q_j,q_{j+1}, q_{j+r}, q_{j+r+1})}\\
   &\leq \frac{2^{3r+6}\cdot r}{Q^3}\sum_{q=1}^{[2Q/T]}\frac{\varphi(q)}{q}.
\end{align*}
Using M\"obius inversion, we have
\[
\frac{\varphi(q)}{q} = \sum_{d|q}\frac{\mu(d)}{d}.
\]
Therefore, 
\begin{align*}
    \sum_{q=1}^{[2Q/T]}\frac{\varphi(q)}{q} &= \sum_{q=1}^{[2Q/T]}\sum_{d|q}\frac{\mu(d)}{d} = \sum_{d\leq [2Q/T]}\frac{\mu(d)}{d}\left\lfloor \frac{[2Q/T]}{d}\right\rfloor\\
    &\leq [2Q/T]\sum_{d\leq [2Q/T]}\frac{\mu(d)}{d^2}+\sum_{d\leq [2Q/T]}\frac{1}{d}\\
    &\leq \frac{12Q}{\pi^2 T}+\frac{T}{Q}+\log\frac{2Q}{T}+\gamma+\frac{T}{2Q}\leq \frac{12Q}{\pi^2 T}+\log\frac{2Q}{T}+2,
\end{align*}
provided that $[2Q/T]\geq 1$. Substituting back, we obtain
\begin{align}
   \Big|S_{r,T}(Q)-S_r(Q) \Big|&\leq \frac{2^{3r+6}\cdot r}{Q^3}\left(\frac{12Q}{\pi^2 T}+\log(2Q/T)+2\right).\label{eq:approximate SrT(Q) by Sr(Q)}
\end{align}

Upon combining \eqref{eq:approximate SrT(Q) by double integral} and \eqref{eq:approximate SrT(Q) by Sr(Q)}, the lemma follows.
\end{proof}

\subsection{Asymptotic formula for \texorpdfstring{$S_r(Q)$}{Sr(Q)}}

Now we are ready to prove Theorem \ref{thm: asymptotic for Sr(Q)}.


  
\begin{proof}[Proof of Theorem \ref{thm: asymptotic for Sr(Q)}.]

  Observe that 
    \begin{align*}
   \bigcup_{T,M\geq 1}\mathscr{D}(T, M) = \mathscr{T} \quad\text{ and }\quad \mathscr{D}(T, M)\subset \mathscr{D}(T_1, M_1)\text{ for $T_1\geq T$ and $M_1\geq M$}.
    \end{align*}
    Therefore, 
    \[
    I_r =\lim_{T_1,M_1\rightarrow\infty} \iint\limits_{\mathscr{D}(T_1,M_1)}f_r(x,y)\ dxdy,
    \]
    and 
    \begin{align*}
     \Big| I_r -  \iint\limits_{ \mathscr{D}(T, M)} f_r(x,y)dx\ dy \Big| = \Big| \lim_{T_1,M_1\rightarrow\infty} \iint\limits_{ \mathscr{D}(T_1, M_1)\backslash \mathscr{D}(T, M)} f_r(x,y)dx\ dy \Big|. 
    \end{align*}
    Recall the definition of $f_r(x,y)$ in \eqref{def:fr(x,y)} and the independence of $\mathscr{D}(T, M)$ from $Q$. For every $N$ with $2N\geq T\geq 2^{r+3}$, Lemma \ref{lem: approximate Q squared Sr(Q)} gives
    \begin{align*}
     \iint\limits_{\mathscr{D}(T_1, M_1)\backslash \mathscr{D}(T, M)} f_r(x,y)dx\ dy &\leq  \pi^2\cdot 2^{3r+6}\frac{T_1^r M_1^2 \log N}{N}+7\pi^2\cdot 2^{3r+6}\frac{M_1(1+4\log N)}{3N}\notag\\
  &\quad+ 2^{3r+6}\frac{M_1\pi^2}{3N}+\frac{2^{3r+8}\pi^2}{3M}+\frac{2^{3r+6}\pi^2 r}{3N}\left(\frac{12N}{\pi^2 T}+\log\frac{2N}{T}+2\right).
   \end{align*}
Since the left side is independent of $N$, we may let $N\rightarrow\infty$ while keeping $T_1,T,M_1,M$ fixed. This gives
\begin{align*}
     \iint\limits_{\mathscr{D}(T_1, M_1)\backslash \mathscr{D}(T, M)} f_r(x,y)dx\ dy &\leq \frac{2^{3r+8}\pi^2}{3M}+\frac{2^{3r+8}r}{T}.
\end{align*}
Now letting $T_1,M_1\rightarrow\infty$, we obtain
\[
\Big| I_r -  \iint\limits_{ \mathscr{D}(T, M)} f_r(x,y)dx\ dy \Big| \leq \frac{2^{3r+8}\pi^2}{3M}+\frac{2^{3r+8}r}{T}.
\]

Choose $M=T=\lceil (Q/\log Q)^{1/(r+3)}\rceil$. Combining with Lemma \ref{lem: approximate Q squared Sr(Q)}, \mbox{for $Q/\log Q\geq 2^{(r+3)^2}$,} we arrive at 
\begin{align*}
|R_r|
&\leq
\left(3\cdot 2^{3r+6}
+3\cdot 2^{3r+8}
+\frac{9\cdot 2^{3r+8}r}{\pi^2}\right)
\frac{\log^{1/(r+3)}Q}{Q^{2+1/(r+3)}}+7\cdot 2^{3r+6}\frac{(1+4\log Q)}{(\log Q)^{1/(r+3)}Q^{3-1/(r+3)}} \\
&\quad+2^{3r+6}\frac{1}{(\log Q)^{1/(r+3)}Q^{3-1/(r+3)}}+\frac{2^{3r+6}r}{Q^3}
\left(\log Q+2\right).
\end{align*}
    
\end{proof}

\section{Mundici's conjecture in short intervals}\label{sec: Section 6}
Let $I=(\alpha,\beta]$ be a subinterval of $(0,1])$. We devote this section to the estimation of $S_0(Q,I)$ defined in \eqref{def:S0(Q,I)},  with explicit error bounds. To this end, we give a proof of Theorem \ref{thm: Mundici Conjecture Short}, which is the generalization of Mundici's original conjecture to short intervals.
\subsection{Preliminary Lemmas}
The general idea of the proof follows that in Section 9 of \cite{BoCoZaha2001}, except that we obtain concrete error bounds for the implied constants. In this subsection, we present several lemmas needed for the asymptotic formula of $S_0(Q,I)$. To do so, similarly to Section \ref{sec: Section 3}, we need to estimate sums of the type
\begin{align}
S_I' = S_I'(f,\Omega) = \sum_{\substack{(a,b)\in\Omega\cap\Z^2\\\gcd(a,b)=1\\\overline{b}\in I_a}} f(a,b),\label{def:S'I}
\end{align}
where $I_q := [q(1-\beta),\ q(1-\alpha))$.

\begin{lem}\textbf{(Effective version of Lemma 1.6 in \cite{BoCoZaha2000})}\label{lem: Lemma 4.2}
For any positive integer $q$, any integers $m$ and $n$, and any subinterval \(I\) of \([1,q]\), denote 
\begin{align}
S_{I}(m,n,q):=\sum_{\substack{x\in I\\(x,q)=1}}e\left(\frac{mx+n\bar{x}}{q}\right).\label{def:SI(m,n,q)}
\end{align}
Then for $q\geq e^3$,
\[
|S_I(m,n,q)|\leq
   (n,q)^{\frac{1}{2}}q^{\frac{1}{2}+\frac{1.0661}{\log \log q}}(2+\log q).
\]
\end{lem}
\begin{proof}
Let $\sigma_0(q)$ denote the number of divisors of $q$. Following the proof of \cite[Lemma 1.6]{BoCoZaha2000}, we have 
\begin{align*}
    |S_{I}(m,n,q)|&\leq \frac{1}{q}\sum^{q-1}_{k=1}\frac{1}{2\|\frac{k}{q}\|}|S(m-k,n,q)|+ \frac{|I|}{q}|S(m,n,q)|\\
    &\leq \sigma_0(q)(n,q)^{\frac{1}{2}}q^{\frac{1}{2}}\left(\frac{1}{2q}\sum^{q-1}_{k=1}\frac{1}{\|\frac{k}{q}\|}+ \frac{|I|}{q}\right)\\
    &\leq \sigma_0(q)(n,q)^{\frac{1}{2}}q^{\frac{1}{2}}(2+\log q),
\end{align*}
where the second inequality follows from the explicit upper bound of $S(m,n,q)$ in \cite[Equa.6]{Estermann}. For \(q\geq e^3\), using the upper bound of $\sigma_0(q)$ in \cite{NicolasRobin}, we obtain
\[
|S_I(m,n,q)|\leq (n,q)^{\frac{1}{2}}q^{\frac{1}{2}+\frac{1.0661}{\log \log q}}(2+\log q).
\]
\end{proof}

\begin{lem}\textbf{(Effective version of Lemma 9 in \cite{BoCoZaha2001})}\label{lem: Lemma 9}
Let $f$ be a $C^1$ function on $\Omega$, where $\Omega\subseteq \R^2$ is still a convex bounded region with rectifiable boundary $\partial\Omega$. 
\begin{align}
S_{f,J}(l,a):=\sum_{\substack{b\in J\\\gcd(a,b)=1}}f(a,b)e\left(\frac{l\bar{b}}{a}\right) \label{def:Sf,J(l,a)}
\end{align}
with \(J\) a bounded interval in $\R$. Then, for \(a\geq e^3\),
\[
|S_{f,J}(l,a)|\leq
2m\|f\|_\infty a^{\frac{1}{2}+\frac{1.0661}{\log \log a}}(2+\log a)\left(\frac{|J|}{a}+1\right)(l,a)^{1/2},
\]
where \(m=m_f\) is an upper bound for the number of intervals of monotonicity of the function~\(J \ni y \mapsto f(a,y)\).
\end{lem}
\begin{proof}
   By Lemma \ref{lem: Lemma 4.2}, for any \(J_0\) subinterval of \([1,a]\) with \(a\geq e^3\), we have 
\begin{align*}
|S_{1,J_0}(l,a)|=|S_{J_0}(0,l,a)|\leq a^{\frac{1}{2}+\frac{1.0661}{\log \log a}}(2+\log a)(l,a)^{1/2}.
\end{align*}
Applying partial summation as in \cite[Lemma 9]{BoCoZaha2001}, we obtain the lemma. 
\end{proof}

\begin{lem}\textbf{(Effective version of Lemma 10 in \cite{BoCoZaha2001})}\label{lem: Lemma 10}
Suppose that \(\Omega\) is a convex subset of the rectangle
\[
[A,A+R]\times[B,B+R]
\]
for some \(A,B\geq e^3\) and \(R\geq 1.\) Then,
\begin{align*}
|S'_I-|I|S'|&\leq \|f\|_{\infty}\frac{(R+1)R}{A}\\
&\quad+4 m\|f\|_{\infty}R(A+R)^{\frac{1}{2}+\frac{2.1322}{\log \log A}}
\Big[2\log (A+R)
+ \log^2 (A+R)
\Big],
\end{align*}
where $S'$ and $S_I'$ are defined in \eqref{def:S'I} and \eqref{def:S'} respectively, and \(m=m_f\) is an upper bound for the number of intervals of monotonicity of the function~\(J \ni y \mapsto f(x,y)\).
\end{lem}
\begin{proof}
Observe that
\[
S'_I=\sum_{\substack{(a,b)\in\Omega\cap\Z^2\\\gcd(a,b)=1}}f(a,b)\sum_{x\in I_a}\frac{1}{a}\sum^a_{l=1}e\left(\frac{l(\bar{b}-x)}{a}\right):= S_1+S_2,
\]
where \(S_1\) is the sum of terms with \(l=a\) and \(S_2\) is the sum of the remaining terms. 

\mbox{Since \(|I_a|=|I|a\), }we then have
\[
S_1=\sum_{\substack{(a,b)\in\Omega\cap\Z^2\\ \gcd(a,b)=1}}f(a,b)\sum_{x\in I_a}\frac{1}{a}\leq \sum_{\substack{(a,b)\in\Omega\cap\Z^2\\ \gcd(a,b)=1}}f(a,b)\frac{1}{a}(|I_a|+1)=|I|S'+ \sum_{\substack{(a,b)\in\Omega\cap\Z^2\\ \gcd(a,b)=1}}f(a,b)\frac{1}{a}.
\]
The second sum on the right side can be bounded by 
\begin{align*}
\sum_{\substack{(a,b)\in\Omega\cap\Z^2\\ \gcd(a,b)=1}}f(a,b)\frac{1}{a} 
&\leq  \|f\|_{\infty}\sum_{\substack{A\leq a\leq A+ R_1\\B\leq b\leq B+R_2}}\frac{1}{a}\#\{b:(a,b)\in\Omega\cap\Z^2,\gcd(a,b)=1\}\\
&\leq \|f\|_{\infty}\sum_{A\leq a\leq A+R}\frac{1}{a}R\leq \|f\|_{\infty}\frac{(R+1)R}{A}.
\end{align*}

Now, it remains to upper bound \(S_2\). By definition,
\[
S_2=\sum_{\substack{(a,b)\in\Omega\cap\Z^2\\\gcd(a,b)=1}}f(a,b)\sum_{x\in I_a}\frac{1}{a}\sum^{a-1}_{l=1}e\left(\frac{l(\bar{b}-x)}{a}\right)=\sum_{a\in \text{pr}_1(\Omega)}\frac{1}{a}\sum^{a-1}_{l=1}\left(\sum_{x\in I_a} e\left(-\frac{lx}{a}\right)\right)S_{f,I'_a}(l,a),
\]
where $\text{pr}_1(\Omega)$ is the projection of $\Omega$ onto the first coordinate and $I_a' =I_a\cap \{b\in\R:(a,b)\in\Omega\}$.
Observe that the sum over $x$ in $S_2$ is a geometric sum, and thus we have
\begin{align*}
|S_2|
&\leq \sum_{A\leq a\leq A+ R}\frac{1}{a}\sum_{l=1}^{a-1}\frac{a}{2\min(l,a-l)}
\cdot |S_{f,I_a'}(l,a)|\\
&\leq \sum_{A\leq a\leq A+R}\frac{1}{a}\sum_{l=1}^{a-1}
\frac{a}{\min(l,a-l)}\cdot
m\|f\|_{\infty}a^{\frac{1}{2}+\frac{1.0661}{\log \log a}}(2+\log a)\left(\frac{|I_a'|}{a}+1\right)(l,a)^{1/2}\\
&\leq 4m\|f\|_{\infty}\sum_{A\leq a\leq A+ R}a^{\frac{1}{2}+\frac{1.0661}{\log \log a}}(2+\log a)\sum_{l=1}^{a-1}\frac{(l,a)^{1/2}}{l},
\end{align*}
where the second inequality follows from Lemma \ref{lem: Lemma 9} and the third follows from the fact \mbox{that $(l,a) = (a-l,a)$.} We can further simplify the innermost sum as 
\begin{align*}
    \sum^{a-1}_{l=1}\frac{(l,a)^{1/2}}{l}\leq \sum_{d|a}d^{-1/2}\log a
    \leq \sigma_0(a)\log a\leq a^{\frac{1.0661}{\log\log a}}\log a.
\end{align*}

Plugging this back into \(|S_2|\), we have 
\begin{align*}
|S_2|
&\leq 4m\|f\|_{\infty}\sum_{A\leq a\leq A+ R}
a^{\frac{1}{2}+\frac{2.1322}{\log \log a}}(2\log a+\log^2 a)\\
&\leq 4 m\|f\|_{\infty}R(A+R)^{\frac{1}{2}+\frac{2.1322}{\log \log A}}
\Big[2\log (A+R)
+ \log^2 (A+R)
\Big].
\end{align*}
Therefore, combining both parts, we have 
\begin{align*}
|S'_I-|I|S'|&\leq \|f\|_{\infty}\frac{(R+1)R}{A}\\
&\quad+4 m\|f\|_{\infty}R(A+R)^{\frac{1}{2}+\frac{2.1322}{\log \log A}}
\Big[2\log (A+R)
+ \log^2 (A+R)
\Big].
\end{align*}
\end{proof}

\subsection{Explicit asymptotic formula for \(S_0(Q,I)\)}
Recall in \eqref{def:S0(Q,I)} the definition of $S_0(Q,I)$:
\[
S_0(Q,I)= \sum_{\gamma_j\in F_I(Q)}(\gamma_{j+1}-\gamma_j)^2= \sum_{\gamma_j\in F_I(Q)}\frac{1}{q^2_jq^2_{j+1}}.
\]

Denote \(T=Q^c\) for small \(c\in (0,1)\). In \cite{BoCoZaha2001}, the constant $c$ is optimized to be $1/10$. We decompose \(S_0(Q,I)\) as the sum of \(T_1(Q,I)+  T_2(Q,I) +T_3(Q,I)\), where 
\begin{align*}
    T_1(Q,I)&=\sum_{\substack{\gamma_j\in F_I(Q)\\q_j,q_{j+1}\geq Q/T}}\frac{1}{q^2_jq^2_{j+1}},\\
    T_2(Q,I)&=\sum_{\substack{\gamma_j\in F_I(Q)\\q_j<Q/T}}\frac{1}{q^2_jq^2_{j+1}},\\
    T_3(Q,I)&=\sum_{\substack{\gamma_j\in F_I(Q)\\q_{j+1}<Q/T}}\frac{1}{q^2_jq^2_{j+1}}.
\end{align*}

Now we introduce the defect constant $c_I$ that depends on the endpoints of the \mbox{interval $I$.} Denote 
\[
c_I=\sum_{q\geq 1}\frac{\#\{a\in qI;(a,q)=1\}-|I|\varphi(q)}{q^2}.
\]
Notice that 
\begin{align*}
\left|\#\{a\in qI; (a,q)=1\}-|I|\varphi(q)\right|&= \Big|\sum_{a\in qI}\sum_{\substack{d|q\\d|a}}\mu(d)-|I|\varphi(q)\Big|\\
&=\Big|\sum_{d|q}\mu(d)\cdot\#\{a\in qI: d|a\}-|I|\varphi(q)\Big|\\
&\leq |\sum_{d|q}\mu(d)|\leq \sigma_0(q),
\end{align*}
where the first inequality in the previous line follows from $\sum_{d|q}\mu(d)/d = \varphi(q)/q$. 
Using the upper bound of $\sigma_0(q)$ in \cite{NicolasRobin}, we can bound  $c_I$ by
\begin{align}
|c_I|\leq\sum_{q\geq 1}\frac{|\#\{a\in qI;(a,q)=1\}-|I|\varphi(q)|}{q^2}\leq \sum_{q\geq 1}\frac{\sigma_0(q)}{q^2}\leq \zeta^2(2)=\frac{\pi^4}{36}.\label{eq:bound for cI}
\end{align}

Now we evaluate the three components of $S_0(Q,I)$. For \(T_1(Q,I)\), apply Lemma~\ref{lem: Lemma 10} with 
\begin{align*}
A=B=\frac{Q}{T}, \qquad &R=\left(1-\frac1T\right)Q,\qquad \Omega = \{(x,y)\in Q\mathscr{T}:\min(x,y)\geq Q/T\},\\
&f(a,b)=\frac1{a^2b^2}, \qquad \text{and\quad}m=1.
\end{align*}
Then, we obtain $
\|f\|_{\infty}\leq T^4/Q^4$ and so
\begin{align}
    T_1(Q,I)-|I|T_1(Q) = R_{I,1},\label{eq:difference from I*T1(Q)}
\end{align}
where \(T_1(Q) := T_1(Q,(0,1])\) and
\begin{align*}
|R_{I,1}|
&\leq \|f\|_{\infty}\frac{(R+1)R}{A}+4m\|f\|_{\infty}R(A+R)^{\frac12+\frac{2.1322}{\log\log A}}
\Big[2\log(A+R)+\log^2(A+R)\Big]\notag\\
&\leq \frac{T^5}{Q^3}\left(1-\frac1T\right)^2
+\frac{T^5}{Q^4}\left(1-\frac1T\right)+4(T^4-T^3)Q^{-\frac{5}{2}+\frac{2.1322}{\log\log Q/T}}
(2\log Q+\log^2 Q).
\end{align*}

We now turn to $T_2(Q,I)$. Following \cite[Eq. (81)]{BoCoZaha2001}, we obtain
\begin{align*}
T_2(Q,I)&\leq \left(1-\frac{1}{T}\right)^{-2}Q^{-2}\sum_{q\leq Q/T}\frac{\#\{a\in qI;(a,q)=1\}}{q^2}\\
&= \left(\frac{1}{Q^2}+\frac{2T-1}{(T-1)^2Q^2}\right)\sum_{q\leq Q/T}\frac{\#\{a\in qI;(a,q)=1\}}{q^2}.
\end{align*}
Moreover, trivially we have 
\[
T_2(Q,I)\geq \frac{1}{Q^2}\sum_{q\leq Q/T}\frac{\#\{a\in qI;(a,q)=1\}}{q^2}.
\]
Hence, combining both upper and lower bounds of \(T_2(Q,I)\), we have 
\begin{align}
\Big|T_2(Q,I)-\frac{1}{Q^2}\sum_{q\leq Q/T}\frac{\#\{a\in qI;(a,q)=1\}}{q^2}\Big|&\leq \frac{(2T-1)}{(T-1)^2Q^2}\sum_{q\leq Q/T}\frac{\#\{a\in qI;(a,q)=1\}}{q^2}\notag\\
&\leq \frac{(2T-1)}{(T-1)^2Q^2}\sum_{q\leq Q/T}\frac{\varphi(q)}{q^2}\notag\\
&\leq \frac{(2T-1)(\log (Q/T) + 1)}{(T-1)^2Q^2}.\label{eq:error term from T2(Q,I)}
\end{align}

Finally, we estimate \(T_3(Q,I)\). Following the arguments in \cite[p. 230]{BoCoZaha2001}, we have 
\begin{align*}
T_3(Q,I)&\leq \left(1-\frac{1}{T}\right)^{-2}Q^{-2}\sum_{q\leq Q/T}\frac{\#\{a\in qI';(a,q)=1\}}{q^2}
\shortintertext{and} T_3(Q,I)&\geq\frac{1}{Q^2}\sum_{q\leq Q/T}\frac{\#\{a\in qI'';(a,q)=1\}}{q^2},
\end{align*}
where 
\[
I'=\left(\alpha,\beta+\dfrac{2}{Q}\right]\quad\text{ and }\quad I''=\left(\alpha+\frac{2}{Q},\beta\right] .
\]
Notice that the interval \(qI'\setminus qI= (q\beta, q\beta +(2q/Q)]\) contains at most one integer for \(Q>2^{1/c}\) \mbox{and \(q\leq Q/T=Q^{1-c}\).} Similar arguments apply to \(qI\setminus qI''\). Thus, for \(Q>2^{1/c}\), we obtain
\begin{align*}
T_3(Q,I) &\leq \left(1-\frac{1}{T}\right)^{-2}Q^{-2}\sum_{q\leq Q/T}\frac{\#\{a\in qI;(a,q)=1\}}{q^2} + \left(1-\frac{1}{T}\right)^{-2}Q^{-2}\sum_{q\leq Q/T}\frac{1}{q^2}\\
&\leq  \left(1-\frac{1}{T}\right)^{-2}Q^{-2}\sum_{q\leq Q/T}\frac{\#\{a\in qI;(a,q)=1\}}{q^2}+ \frac{\pi^2}{6}\left(1-\frac{1}{T}\right)^{-2}Q^{-2},
\end{align*}
and 
\begin{align*}
    T_3(Q,I)&\geq  \frac{1}{Q^2}\sum_{q\leq Q/T}\frac{\#\{a\in qI;(a,q)=1\}}{q^2}- \frac{1}{Q^2}\sum_{q\leq Q/T}\frac{1}{q^2} \\
&\geq \frac{1}{Q^2}\sum_{q\leq Q/T}\frac{\#\{a\in qI;(a,q)=1\}}{q^2}-\frac{\pi^2}{6Q^2}.
\end{align*}
Hence, together with \eqref{eq:error term from T2(Q,I)}, we obtain
\[
\Big|T_3(Q,I)- \frac{1}{Q^2}\sum_{q\leq Q/T}\frac{\#\{a\in qI;(a,q)=1\}}{q^2}\Big|\leq \frac{(2T-1)(\log (Q/T) + 1)}{(T-1)^2Q^2}+ \frac{\pi^2T^2}{6(T-1)^2Q^2}.
\]

Thus, combining the estimations for \(T_2(Q,I)\) and \(T_3(Q,I)\), we have 
\begin{align}
    T_2(Q,I)+T_3(Q,I)= \frac{2}{Q^2}\sum_{q\leq Q/T}\frac{\#\{a\in qI;(a,q)=1\}}{q^2}+ R_{I,2},\label{eq:T2+T3+error term}
\end{align}
where
\[
|R_{I,2}|\leq  \frac{(4T-2)(\log (Q/T) + 1)}{(T-1)^2Q^2}+ \frac{\pi^2T^2}{6(T-1)^2Q^2}.
\]

Define $T_2(Q) := T_2(Q,(0,1])$ and $T_3(Q) := T_3(Q,(0,1])$. Using \eqref{eq:T2+T3+error term} with $I=(0,1]$, we have
\begin{align}
T_2(Q)+T_3(Q)= \frac{2}{Q^2}\sum_{q\leq Q/T}\frac{\varphi(q)}{q^2} +  R_{I,3},\label{eq:T2+T3}
\end{align}
where
\[
|R_{I,3}|\leq \frac{(4T-2)(\log (Q/T) + 1)}{(T-1)^2Q^2}.
\]
Notice that the $\pi^2/6$-term is not included in $R_{I,3}$, because when $I$ is the full interval, $qI'\setminus qI$ and $qI\setminus qI''$ don't contain any integer for $Q>2^{1/c}$. Recall that $T=Q^c$.
 Substituting $c=1/10$ and combining all restrictions for $Q$, \eqref{eq:bound for cI}, \eqref{eq:difference from I*T1(Q)}, \eqref{eq:T2+T3+error term}, and \eqref{eq:T2+T3}, we obtain that for \(Q> 1024\),
\begin{align}
S_0(Q,I)&=T_1(Q,I)+T_2(Q,I)+T_3(Q,I)\notag\\
&= |I|S_0(Q) + \frac{2}{Q^2}\sum_{q\leq Q/T}\frac{\#\{a\in qI;(a,q)=1\}-|I|\varphi(q)}{q^2} +R_{I,1}+R_{I,2}+|I|R_{I,3}\notag\\
&=|I| S_0(Q) + R_{I,4},\label{eq:S0(Q,I) final asymptotic}
\end{align}
where
\begin{align*}
    |R_{I,4}|&\leq  \frac{\pi^4}{18Q^2}+Q^{-5/2}\left(1-Q^{-1/10}\right)^2\\
    &\quad+4Q^{-21/10+2.1322/(\log\log Q-0.1054)}(1+Q^{-1/10})
(2\log Q+\log^2Q)\\
&\quad+\frac{(|I|+1)(4Q^{1/10}-2)((9\log Q)/10+1)}
{(Q^{1/10}-1)^2Q^2}+\frac{\pi^2Q^{-9/5}}{6(Q^{1/10}-1)^2}.
\end{align*}
Note that the constant $-0.1054$ comes from \mbox{bounding $\log\log (Q/T)\geq \log\log Q-0.1054$.}

\subsection{Proof of Theorem \ref{thm: asymptotic of S0(Q,I) with error order -2}}
 Substituting the explicit bound of \(S_0(Q)\) from \cite{ViAnji} into \eqref{eq:S0(Q,I) final asymptotic}, we arrive at the desired result.
 
 \begin{rem}
     In \cite{BoCoZaha2001}, the analogue of our Theorem \ref{thm: asymptotic of S0(Q,I) with error order -2} is Theorem 2. Their Theorem 2 includes terms of order $Q^{-2}$ in the main term rather than in the error term. This is not achievable in our explicit setting. Indeed, when estimating $T_2(Q,I)$ and $T_3(Q,I)$, we don't have an explicit upper bound on $Q$ which guarantees that the interval 
     \[
  qI'\setminus qI= \left(q\beta, q\beta +\dfrac{2q}{Q}\right]
     \]   
      contains no integer. This is precisely where the error term of order $Q^{-2}$ arises. Moreover, since the contribution from terms involving the defect $c_I$ is absorbed in an error term of order $Q^{-2}$, the restriction that the endpoints of the subinterval $I$ be rational can be removed.
 \end{rem}
\section*{Funding}
A.D. is supported by the Shaff--Andrews Fellowship, Department of Mathematics, University of Illinois Urbana-Champaign.

\end{document}